\documentclass[11pt,a4paper]{amsart}

\pdfoutput=1

\usepackage{amssymb}
\usepackage{amsmath}
\usepackage{amsthm}
\usepackage{amsbsy}
\usepackage{enumerate}
\usepackage{xcolor}
\usepackage{esint}
\usepackage{comment}
\usepackage{graphicx}

\usepackage[T1]{fontenc}
\usepackage[utf8]{inputenc}
\usepackage[english]{babel}

\theoremstyle{plain}
\newtheorem{theorem}{Theorem}[section]
\newtheorem{proposition}[theorem]{Proposition}
\newtheorem{lemma}[theorem]{Lemma}

\theoremstyle{definition}

\newtheorem{remark}[theorem]{Remark}
\theoremstyle{remark}
\newtheorem*{discussion}{Discussion}

\numberwithin{equation}{section}

\renewcommand*{\P}{\mathbb{P}}
\newcommand*{\E}{\mathbb{E}}
\newcommand*{\R}{\mathbb{R}}

\newcommand*{\Var}{\operatorname{Var}}

\newcommand*{\eps}{\varepsilon}

\newcommand*{\EE}{\mathbb E}
\newcommand*{\PP}{\mathbb P}

\newcommand*{\bbN}{\mathbb N}

\newcommand*{\bbR}{\mathbb R}

\newcommand*{\cA}{\mathcal A}

\newcommand*{\cC}{\mathcal C}

\newcommand*{\cG}{\mathcal G}

\newcommand*{\cI}{\mathcal I}

\newcommand*{\cL}{\mathcal L}

\DeclareMathOperator{\sgn}{sgn}
\DeclareMathOperator{\arsinh}{arsinh}
\DeclareMathOperator{\sech}{sech}

\renewcommand*{\doteq}{:=}

\begin{document}
\title{Variance reduction for discretised diffusions via regression}

\author{Denis Belomestny}
\address{University of Duisburg-Essen, Essen, Germany and IITP RAS, Moscow, Russia}
\email{denis.belomestny@uni-due.de}

\author{Stefan H\"afner}
\address{PricewaterhouseCoopers GmbH, Frankfurt, Germany}
\email{stefan.haefner@de.pwc.com}

\author{Tigran Nagapetyan}
\address{University of Oxford, UK}
\email{nagapetyan@stats.ox.ac.uk}

\author{Mikhail Urusov}
\address{University of Duisburg-Essen, Essen, Germany}
\email{mikhail.urusov@uni-due.de}

\thanks{The work of Denis Belomestny is supported
by the Russian Science Foundation project 14-50-00150.}

\begin{abstract}
In this paper we present a novel approach towards  variance reduction for discretised diffusion processes.  The proposed approach involves specially constructed control variates and allows for a significant reduction in the variance for the terminal functionals. In this way the complexity order of the standard Monte Carlo algorithm
($\varepsilon^{-3}$ in the case of a first order scheme
and $\varepsilon^{-2.5}$ in the case of a second order scheme)
can be reduced down to $\varepsilon^{-2+\delta}$ for  any $\delta\in [0,0.25)$ with $\varepsilon$ being the precision to be achieved. These theoretical results are illustrated by several numerical examples. 


\end{abstract}

\maketitle

\section{Introduction}

Let \(T>0\) be a fixed time horizon.
Consider a $d$-dimensional diffusion process
$(X_t)_{t\in[0,T]}$
defined by the It\^o stochastic differential equation
\begin{align}
\label{x_sde}
dX_t
=\mu(X_t)\,dt
+\sigma(X_t)\,dW_{t},
\quad X_{0}=x_0\in{\mathbb{R}^d},
\end{align}
for Lipschitz continuous functions
\(\mu\colon\mathbb{R}^d\to\mathbb{R}^d\)
and
\(\sigma\colon\mathbb{R}^d
\to\mathbb{R}^{d\times m}\),
where \((W_t)_{t\in[0,T]}\)
is a standard \(m\)-dimensional Brownian motion.
Recall that, since $\mu$ and $\sigma$ are Lipschitz,
the stochastic differential equation~\eqref{x_sde}
has a strong solution, and pathwise uniqueness holds.
Suppose we want to find a continuous function 
\begin{align*}
u\colon[0,T]\times\mathbb{R}^d\to\mathbb{R},
\end{align*}
which has a continuous first derivative with respect to
the time argument
and continuous first and second derivatives
with respect to the components of the space argument
on $[0,T)\times\mathbb{R}^d$
such that it solves the partial differential equation 
\begin{align}
\label{eq:Cauchy_prob}
\frac{\partial u}{\partial t}+\mathcal{L} u&=0 \quad \mbox{ on }  [0,T)\times \mathbb{R}^d,\\
\label{eq:term_cond}
u(T,x)&=f(x)\quad 	\mbox{ for } x\in \mathbb{R}^d,
\end{align}
where $f$ is a given continuous function on~$\mathbb{R}^d$.
Here and in what follows, $t$ denotes the time argument,
$x$~denotes the space argument of~$u$,
and \(\mathcal{L}\) is the differential operator associated with the equation~\eqref{x_sde}:
\begin{align*}
(\mathcal{L}u)(t,x)
\doteq\sum_{i=1}^{d}\mu^{i}(x)\frac{\partial u}{\partial x_{i}}(t,x)
+\frac{1}{2}\sum_{i,j=1}^{d}(\sigma\sigma^\top)^{ij}(x)\frac{\partial^{2} u}{\partial x_{i}\partial x_{j}}(t,x),
\end{align*}
where $\sigma^\top$ denotes the transpose of~$\sigma$,
and the components of $\mu$ and $\sigma\sigma^\top$
(and later the ones of~$\sigma$)
are denoted by superscripts.
Under appropriate conditions on
$\mu$, $\sigma$ and~$f,$ there is a solution
of the Cauchy
problem \eqref{eq:Cauchy_prob}--\eqref{eq:term_cond},
which is unique in the class of solutions
satisfying certain growth conditions,
and it has the following Feynman-Kac
stochastic representation
\begin{align*}
u(t,x)=\E[ f(X_{T}^{t,x})]
\end{align*}
(see Section~5.7 in~\cite{karatzas2012brownian}),
where $X^{t,x}$ denotes the solution started
at time $t$ in point~$x$. Moreover it holds
(see e.g.\ Newton~\cite{newton1994variance})
\begin{align*}
\E[f(X_{T}^{0,x})|X_{t}^{0,x}]=u(t,X_{t}^{0,x}), \quad \mbox{ a.s. }
\end{align*}
for \(t\in [0,T]\) and 
\begin{align}
\label{repr_contr_var}
f(X_{T}^{0,x})=\E [f(X_{T}^{0,x})]+M^{*}_T, \quad \mbox{ a.s. }
\end{align}
with 
\begin{align}
\label{MT_cont}
M^*_T\doteq \int_{0}^{T}\nabla_x u (t,X_{t}^{0,x})\,\sigma(X_{t}^{0,x})\,dW_{t}
=\int_{0}^{T}\sum_{i=1}^d \frac{\partial u}{\partial x_i} (t,X_{t}^{0,x})\sum_{j=1}^m\sigma^{ij}(X_{t}^{0,x})\,dW^j_{t}.
\end{align}
The standard Monte Carlo (SMC) approach for computing \(u(0,x)\) at a fixed point \(x\in \bbR^d\) consists of three steps. First an approximation \(\overline{X}_T\) for \(X^{0,x}_T\)  is constructed via a time discretisation of the equation \eqref{x_sde} (we refer to \cite{KP} for a nice overview of various discretisation schemes). Next \(N_0\) independent copies of the approximation \(\overline{X}_T\)  are generated and finally  a Monte Carlo estimate \(V_{N_0}\) is defined as an average of the values of \(f\) at simulated points:
\begin{align*}
V_{N_0}\doteq\frac{1}{N_0}\sum_{i=1}^{N_0} f\Bigl(\overline{X}_T^{(i)}\Bigr).
\end{align*} 
In the computation of $u(0,x)=\EE [f(X^{0,x}_T)]$
by the SMC approach
there are two types of error inherent:
a discretisation error
$\EE [f(X^{0,x}_T)]-\EE [f(\overline X_{T})]$
and a Monte Carlo (statistical) error,
which results from the substitution of
$\EE [f(\overline{X}_{T})]$
with the sample average
$V_{N_0}$.
The aim of variance reduction methods is to reduce the statistical error.  For example, in the so-called control variate variance reduction approach
one looks for a random variable
\(\xi\) with \(\EE \xi=0\) such that
the variance of the difference
\(f(\overline{X}_{T})-\xi\) is minimised, i.e., 
\begin{align*}
%\label{var_red_gen}
\Var[f(\overline{X}_{T})-\xi]\to\min\text{ under } \EE\xi=0.
\end{align*}
The use of control variates for solving \eqref{x_sde} via Monte Carlo path simulation approach was initiated by Newton~\cite{newton1994variance} and further developed in Milstein and Tretyakov~\cite{milstein2009practical}. In fact, the construction of the appropriate control variates  in the above two papers  essentially relies on the identity   \eqref{repr_contr_var} implying that 
the zero-mean random variable $M^*_{T}$
can be viewed as an optimal control variate,
since 
\begin{align*}
\Var[f(X^{0,x}_{T})-M^*_{T}]
=\Var[\EE f(X^{0,x}_{T})]=0.
\end{align*}
Let us note that it would be desirable to have
a control variate  reducing the variance
of $f(\overline{X}_{T})$ rather than the one of $f(X^{0,x}_T)$
because we simulate from the distribution of $f(\overline{X}_{T})$
and not from the one of $f(X^{0,x}_T)$. Moreover, the control variate \(M^*_{T}\) cannot be directly computed, since the function \(u(t,x)\) is unknown. This is why Milstein and Tretyakov~\cite{milstein2009practical} proposed to use regression for getting a preliminary  approximation for  \(u(t,x)\) in a first step.

The contribution of our work is threefold. First, we propose an approach for the construction of control variates which reduce the variance of   \(f(\overline{X}_{T}).\)  As a by-product
our control variates can be  computed in a rather simple way. More importantly, we are able to achieve a higher order convergence of  the resulting variance to zero, which in turn leads to a significant complexity reduction as compared to the SMC algorithm. Other prominent examples of Monte Carlo algorithms with this property are multilevel Monte Carlo (MLMC) algorithm of \cite{giles2008multilevel} and a quadrature-based algorithm of \cite{muller2015complexity}.
%Let us note here that if we  use a discretised version of~\eqref{MT_cont} with some plug-in estimate for \(u,\) complexity reduction does not take place.
Our approach becomes especially simple
in the case of the so-called weak approximation schemes,
i.e.\ the schemes,
where simple random variables
are used in place of Brownian increments.
In recent years weak approximation schemes
became quite popular. 
The weak Euler scheme is a first order scheme with
weak order of convergence
\(\alpha=1\), and has been studied by many researchers. Milstein~\cite{mil1979method} showed the first order convergence of the weak Euler scheme.
%The fact that the same weak convergence rate of the Euler scheme also holds for certain irregular functions  under a H\"ormander type condition was proved by Bally and Talay \cite{bally1995euler} using Malliavin calculus.
The It\^o-Taylor (weak Taylor) high-order scheme is a natural extension of the weak Euler scheme. In the  diffusion case, some new discretization schemes (also called Kusuoka type schemes) which are of order \(\alpha\geq 2\) without the Romberg extrapolation have been introduced by Kusuoka~\cite{kusuoka2004approximation}, Lyons and Victoir~\cite{lyons2004cubature}, Ninomiya and Victoir~\cite{ninomiya2008weak}, and Ninomiya and Ninomiya~\cite{ninomiya2009new}. A general class of weak approximation methods, comprising  many well-known discretisation schemes, was constructed in Tanaka and Kohatsu-Higa~\cite{tanaka2009operator}. The main advantage of the weak approximation schemes is that simple discrete random variables can be used to approximate multiple Wiener integrals arising in higher order schemes.

Summing up, we propose a new regression-type approach for the construction of higher order control variates.
It takes advantage of the smoothness in $\mu$, $\sigma$ and $f$ (which is needed for nice convergence properties of regression methods) in order to significantly reduce the variance of the random variable $f(\overline{X}_{T})$.

The paper is organised as follows.
Section~\ref{sec:2} contains a discrete-time analogue of the Clark-Ocone formula for schemes with Gaussian innovations,
which provides the basis for constructing control variates via regression methods.
The corresponding formulas for weak approximation schemes are discussed in Section~\ref{sec:3},
where  the schemes of first and second order are analysed in detail.
Section~\ref{sec:4} describes a generic regression algorithm for the construction of control variates.
Error bounds in the generic algorithm depend on a particular implementation, i.e.\ on the choice of basis functions for regressions.
For the specific choice of piecewise polynomial regression, the error bound is presented in Section~\ref{sec:4.5} and the complexity analysis of the algorithm in Section~\ref{sec:5}.
Section~\ref{sec:6} is devoted to a simulation study.
Finally, all proofs are collected in Section~\ref{sec:proofs}.

\section{Control variates for schemes with Gaussian increments}
\label{sec:2}
To begin with, we introduce some notations,
which will be frequently used in the sequel.
Throughout the paper
$\bbN_0\doteq\bbN\cup\{0\}$
denotes the set of nonnegative integers,
$J\in\bbN$ denotes the time discretisation parameter,
we set $\Delta\doteq T/J$
and consider discretisation schemes
defined on the grid
$\{j\Delta:j=0,\ldots,J\}$.
We recall that $X$ in~\eqref{x_sde}
is $d$-dimensional and $W$ in~\eqref{x_sde}
is $m$-dimensional for some fixed $d,m\in\bbN$.
For $j\in\{1,\ldots,J\}$, we define
$\Delta_j W\doteq W_{j\Delta}-W_{(j-1)\Delta}$,
and by $W^i$ we denote the $i$-th component
of the vector $W$. Finally, for $k\in\bbN_0$,
$H_k\colon\bbR\to\bbR$
stands for the (normalised) $k$-th Hermite polynomial,~i.e.
\begin{align*}
  H_k(x) 
  \doteq 
  \frac{(-1)^k}
    {\sqrt{k!}}
  e^{\frac{x^2}{2}}
  \frac{d^k}{dx^k}e^{-\frac{x^2}{2}},
  \quad x\in\bbR.
\end{align*}
Notice that $H_0\equiv1$.
To motivate a general construction
of optimal control variates,
let us first look at an example.

\subsection{Motivating example}
Consider a simple one-dimensional SDE 
\begin{align*}
dX_{t}=\sigma X_{t}dW_{t},\quad t\in[0,T],
\end{align*}
with   $X_{0}=x_{0}$, and its Euler discretisation
$(X_{\Delta,j\Delta})_{j=0,\ldots,J}$,
where $X_{\Delta,0}=x_0$ and
\begin{align*}
X_{\Delta,j\Delta}=X_{\Delta,(j-1)\Delta}(1+\sigma\Delta_j W),\quad j=1,\ldots,J.
\end{align*}
Suppose that we would like to approximate the quantity $V\doteq\E[X^2_T].$
It is easy to see that \(\EE\bigl[X_{\Delta, J\Delta}^{2}\bigr]=x_{0}^{2}\,(1+\sigma^{2}\Delta)^{J}\)
 and using a telescopic sum trick, we derive
\begin{align*}
X_{\Delta, J\Delta}^{2}-\EE\left[X_{\Delta, J\Delta}^{2}\right]=\sum_{j=1}^{J}\left(X_{\Delta, j\Delta}^{2}(1+\sigma^{2}\Delta)^{J-j}-X_{\Delta, (j-1)\Delta}^{2}(1+\sigma^2\Delta)^{J-j+1}\right).
\end{align*}
Since $\Delta_j W=\frac{X_{\Delta, j\Delta}-X_{\Delta, (j-1)\Delta}}{\sigma X_{\Delta, (j-1)\Delta}},$ we get
\begin{align*}
X_{\Delta, j\Delta}^{2}-X_{\Delta, (j-1)\Delta}^{2}(1+\Delta\sigma^{2}) 
 & =  2\sigma X_{\Delta, (j-1)\Delta}^{2}\Delta_j W+\sigma^{2}X_{\Delta, (j-1)\Delta}^{2}\left(\Delta_j W^{2}-\Delta\right).
\end{align*}
As a result 
\begin{align}
X_{\Delta, J\Delta}^{2}-\EE\left[X_{\Delta, J\Delta}^{2}\right] 
 = \sum_{j=1}^{J}\left(
 a_{j,1}(X_{\Delta, (j-1)\Delta})H_{1}\left(\frac{\Delta_j W}{\sqrt{\Delta}}\right)+
 a_{j,2}(X_{\Delta, (j-1)\Delta})H_{2}\left(\frac{\Delta_j W}{\sqrt{\Delta}}\right)\right)\label{eq:decomp_var}
\end{align}
with $a_{j,1}(y)=2\sigma\sqrt{\Delta}y^{2}(1+\sigma^{2}\Delta)^{J-j}$
and $a_{j,2}(y)=\sqrt{2}\sigma^{2}\Delta y^{2}(1+\sigma^{2}\Delta)^{J-j}.$
Notice that representation~\eqref{eq:decomp_var} has a very simple form. Furthermore, the coefficients
$a_{j,1}$ and $a_{j,2}$ can be represented as conditional expectations
\begin{align*}
a_{j,k}(X_{\Delta, (j-1)\Delta})=\EE\left[\left.X_{\Delta, J\Delta}^{2}\,H_{k}\left(\frac{\Delta_j W}{\sqrt{\Delta}}\right)\right|X_{\Delta, (j-1)\Delta}\right], \quad k=1,2.
\end{align*}
Thus, the control variate 
\begin{align}
\label{eq:exmpl_cv}
M_{\Delta,J\Delta}\doteq\sum_{j=1}^{J}\sum_{k=1}^{2} a_{j,k}(X_{\Delta, (j-1)\Delta})\,H_{k}(\Delta_{j} W/\sqrt{\Delta}),
\end{align}
is a perfect control variate, as it satisfies \(\Var[X_{\Delta, J\Delta}^{2}-M_{\Delta,J\Delta}]=0.\)  
The above example encourages us to look for  control variates in the form~\eqref{eq:exmpl_cv}, where the coefficients $a_{k,j}(x)$ have the form of conditional expectations,  which in turn can be computed by regression methods. As we will see in the next sections, such perfect control variates can be constructed in the general case.

\begin{discussion}
The control variate in~\eqref{eq:exmpl_cv}
is a sum over all time steps.
At this point it is, therefore, unclear
whether the variance reduction achieved in the proposed method
outweighs the additional computational work
required to implement such a control variate.
After the detailed description of our algorithm
we will present the complexity analysis, which shows that,
given the precision $\varepsilon$ to be achieved,
implementing such a control variate
results in less total computational work,
provided several parameters are chosen a proper way.
\end{discussion}

\subsection{Control variate construction}
Let us consider a scheme,
where $d$-dimensional approximations
$X_{\Delta,j\Delta}$,
$j=0,\ldots,J$, satisfy $X_{\Delta,0}=x_0$ and
\begin{align}\label{eq:defX}
    X_{\Delta, j\Delta} 
    = 
    \Phi_{\Delta}\left(
       X_{\Delta, (j-1)\Delta},  \frac{\Delta_{j} W}{\sqrt{\Delta}}
    \right)
\end{align}
for some Borel measurable functions
$\Phi_{\Delta}\colon\bbR^{d+m}\to\bbR^d$
(clearly, the Euler scheme is a special case of this setting).

\begin{theorem}\label{thm:ChaosDecompNum}
Let $f\colon \R^d \rightarrow \R$ be
a Borel measurable function
such that $\EE | f(X_{\Delta, T}) |^2 < \infty$.
Then the following representation holds
\begin{align}
\label{eq:2909a1}
f(X_{\Delta,T})=\E[f(X_{\Delta,T})]+  
\sum_{j=1}^{J}
\sum_{k\in\bbN_0^m\setminus\{0\}}
a_{j,k}(X_{\Delta, (j-1)\Delta})
\prod_{r=1}^m
H_{k_r}\left(\frac{\Delta_{j} W^{r}}{\sqrt{\Delta}}\right),
\end{align}
where $k=(k_1,\ldots,k_m)$ and $0=(0,\ldots,0)$
(in the second summation),
and the coefficients
$a_{j,k}\colon\bbR^d\to\bbR$
are given by the formula
\begin{align*}
a_{j,k}(x)=\EE\left[ f(X_{\Delta,T})
\prod_{r=1}^m H_{k_r}
\left(\frac{\Delta_{j} W^{r}}{\sqrt{\Delta}}\right)
\bigg|\,X_{\Delta, (j-1)\Delta}=x \right],
\end{align*}
for all $j\in \{1,\ldots,J\}$
and $k\in\bbN_0^m\setminus\{0\}$. 
\end{theorem}

\begin{remark}
(i) Representation~\eqref{eq:2909a1}
can be viewed as a discrete-time analogue
of the Clark-Ocone formula.
See e.g.~\cite{AAO} (Gaussian increments),
\cite{PrivaultSchoutens} (Bernoulli increments)
and the references therein for representations of similar types.
Our form~\eqref{eq:2909a1} is aimed at
constructing control variates via regression methods.

(ii) A comparison
of~\eqref{eq:exmpl_cv} and~\eqref{eq:2909a1}
gives rise to the question whether
our motivating example fits the framework~\eqref{eq:2909a1}.
The answer is affirmative:
a straightforward calculation using the facts that
$f(x)=x^2$ in the motivating example and that,
for $k\ge3$, $H_k(\Delta_j W/\sqrt{\Delta})$ is orthogonal
to all polynomials of $\Delta_j W$ of degree two
reveals that $a_{j,k}\equiv0$ whenever $k\ge3$
in the situation of our motivating example.
\end{remark}

\begin{discussion}
Representation~\eqref{eq:2909a1} shows that
the random variable
\begin{eqnarray}
\label{eq:wiener_pcv}
M_{\Delta,T}\doteq
\sum_{j=1}^{J}
\sum_{k\in\bbN_0^m\setminus\{0\}}
a_{j,k}(X_{\Delta, (j-1)\Delta})
\prod_{r=1}^m
H_{k_r}\left(\frac{\Delta_{j} W^{r}}{\sqrt{\Delta}}\right)
\end{eqnarray}
is a perfect control variate for the functional \(f(X_{\Delta, T})\),
i.e.\ \(\Var[f(X_{\Delta, T})-M_{\Delta,T}]=0\).
In order to be able to use this control variate, we need to truncate the summation (over \(k\)) in~\eqref{eq:wiener_pcv} and study the order of the corresponding truncation error. However, as we will see in the next section, we can avoid this problem by using the so-called weak approximation schemes, where the Brownian motion increments
in~\eqref{eq:defX} are replaced by simple discrete-valued random variables.
\end{discussion}

\section{Schemes with discrete random variables in the increments}
\label{sec:3}
In this section we derive the analogue of
representation~\eqref{eq:2909a1}
for the case of weak approximation schemes. In order to be more concise, we focus on the weak schemes of first and second order.

\subsection{First order schemes}
In this subsection we treat weak schemes of order~$1$.
Let us consider a scheme,
where $d$-dimensional approximations
$X_{\Delta,j\Delta}$,
$j=0,\ldots,J$, satisfy $X_{\Delta,0}=x_0$ and
\begin{align}
\label{eq:scheme_structure_md}
X_{\Delta,j\Delta}=
\Phi_{\Delta}(X_{\Delta,(j-1)\Delta},\xi_j),
\quad j=1,\ldots,J,
\end{align}
for some functions
$\Phi_{\Delta}\colon\bbR^{d+m}\to\bbR^d$,
with $\xi_j=(\xi_j^1,\ldots,\xi_j^m)$, $j=1,\ldots,J$,
being $m$-dimensional
iid random vectors with iid coordinates
such that
\begin{align*}
\PP\left(\xi_j^k=\pm1\right)=\frac12, \quad k=1,\ldots,m.
\end{align*}
A particular case is the Euler weak scheme (also called the
\emph{simplified weak Euler scheme}
in \cite[Section~14.1]{KP})
of order~1, which is given by
\begin{align}
\label{eq:PhiK=1}
\Phi_{\Delta}(x,y)
=x+\mu(x)\,\Delta+\sigma(x)\,y\,\sqrt{\Delta}.
\end{align}

\begin{theorem}
\label{th:weak_md01}
The following representation holds
\begin{align}
\label{eq:repr02}
f(X_{\Delta,T})=\EE f(X_{\Delta,T})
+\sum_{j=1}^J
\sum_{r=1}^m
\sum_{1\le s_1<\ldots<s_r\le m}
a_{j,r,s}(X_{\Delta,(j-1)\Delta})
\prod_{i=1}^r \xi_j^{s_i},
\end{align}
where we use the notation
$s=(s_1,\ldots,s_r)$.
Moreover, the coefficients
$a_{j,r,s}\colon\bbR^d\to\bbR$
can be computed by the formula
\begin{align}
\label{eq:coef05}
a_{j,r,s}(x)
=\EE\left[\left.
f(X_{\Delta,T}) \prod_{i=1}^r \xi_j^{s_i} 
\,\right|\,
X_{\Delta,(j-1)\Delta}=x
\right]
\end{align}
for all $j$, $r$, and $s$ as in~\eqref{eq:repr02}.
\end{theorem}

The next proposition shows the properties of the simplified Euler scheme combined with the control variate 
\begin{align}
\label{eq:2909a2}
M^{(1)}_{\Delta,T}\doteq
\sum_{j=1}^J
\sum_{r=1}^m
\sum_{1\le s_1<\ldots<s_r\le m}
a_{j,r,s}(X_{\Delta,(j-1)\Delta})
\prod_{i=1}^r \xi_j^{s_i}, 
\end{align}
where the coefficients
$a_{j,r,s}(x)$ are given by~\eqref{eq:coef05}.
It is a combination of the above Theorem~\ref{th:weak_md01}
together with Theorem~2.1 in~\cite{MilsteinTretyakov:2004}.

\begin{proposition}
%\label{prop:Euler:CV}
Assume that $\mu$ and $\sigma$ in~\eqref{x_sde} are Lipschitz continuous with components $\mu^i,\,\sigma^{i,r}\colon \R^d\to\R$,
$i=1,\ldots,d$, $r=1,\ldots,m$,
being $4$ times continuously differentiable
with their partial derivatives of order up to $4$
having polynomial growth.
Let $f\colon\R^d\to\R$ be $4$ times continuously differentiable with 
partial derivatives of order up to $4$
having polynomial growth.
Provided that \eqref{eq:PhiK=1} holds
and that, for sufficiently large $p\in\mathbb N$,
the expectations $\EE |X_{\Delta,j\Delta}|^{2p}$
are uniformly bounded in $J$ and $j=0,\ldots,J$,
we have for this
``simplified weak Euler scheme''
\begin{align*}
\left|\E\left[f(X_T) - f(X_{\Delta,T})\right]\right|\le c\Delta,
\end{align*}
where the constant $c$ does not depend on $\Delta$. Moreover, it holds
$\Var\left[f(X_{\Delta,T}) - M^{(1)}_{\Delta,T}\right]=0.$
\end{proposition}

\begin{discussion}
In order to use the control variate
$M^{(1)}_{\Delta,T}$
in practice, we need to estimate the unknown coefficients
$a_{j,r,s}$.
Thus, practically implementable control variates
$\tilde{M}^{(1)}_{\Delta,T}$
have the form~\eqref{eq:2909a2}
with some estimated functions
$\tilde{a}_{j,r,s}\colon\bbR^d\to\bbR$.
Notice that they remain valid control variates,
i.e.\ we still have $\EE\bigl[\tilde{M}^{(1)}_{\Delta,T}\bigr]=0$,
which is due to the martingale transform
structure\footnote{\label{ft:19102016a1}This phrase means
that the discrete-time process
$\tilde M=(\tilde M_l)_{l=0,\ldots,J}$, where $\tilde M_0=0$ and
$\tilde M_l$ is defined like the right-hand side of~\eqref{eq:2909a2}
but with
$\sum_{j=1}^J$ being replaced by $\sum_{j=1}^l$
and $a_{j,r,s}$ by $\tilde a_{j,r,s}$
is a martingale, which is a straightforward calculation.}
in~\eqref{eq:2909a2}.
\end{discussion}

%\begin{comment}
%Analogously one can derive the representation.
%\begin{theorem}
%\label{th:weak_md02}
%We have the following representation
%\begin{align}
%\label{eq:repr03}
%f(X_{\Delta,T})=\EE f(X_{\Delta,T})
%+\sum_{j=1}^J
%\sum_{r=1}^m
%a_{j,r}(X_{\Delta,(j-1)\Delta},(\xi_j^i)_{i=1}^{r-1})\xi_j^r,
%\end{align}
%where the coefficients
%$a_{j,r}\colon\bbR^{d+r-1}\to\bbR$
%can be computed by the formula
%\begin{align}
%\label{eq:coef06}
%a_{j,r}(x,(y^i)_{i=1}^{r-1})
%=\EE\left[\left.
%f(X_{\Delta,T}) \xi_j^r
%\,\right|\,
%X_{\Delta,(j-1)\Delta}=x,
%(\xi_j^i)_{i=1}^{r-1}=(y^i)_{i=1}^{r-1}
%\right]
%\end{align}
%for
%$j=1,\ldots,J$
%and
%$r=1,\ldots,m$.
%\end{theorem}
%\paragraph{Discussion}
%Theorems~\ref{th:weak_md01}
%and~\ref{th:weak_md02}
%suggest two perfect control variates
%in the multidimensional case
%for scheme~\eqref{eq:scheme_structure_md}.
%Compared with the control variate
%based on~\eqref{eq:repr02},
%the control variate based
%on~\eqref{eq:repr03} contains
%a smaller number of coefficients
%to be computed,
%but these coefficients
%are functions of a greater number of variables.
%\end{comment}

\subsubsection{Computation of coefficients}\label{subsubsec:311}
Coefficients~\eqref{eq:coef05} can be directly computed \mbox{using} various regression algorithms as discussed in Section~\ref{sec:4}. From a computational point of view it is sometimes advantageous to look for another representation  which only involves a regression over one time step (note that in~\eqref{eq:coef05} regression should be performed over \(J-j+1\) time steps).  To this end, we introduce the functions
\begin{align}
\label{tower:q}
q_j(x)\doteq\EE[f(X_{\Delta, T})|X_{\Delta,j\Delta}=x].
\end{align}
The next proposition contains backward recursion formulas
for the functions $q_j$ as well as the expressions for the coefficients~\eqref{eq:coef05} in terms of \(q_j,\) \(j=1,\ldots,J\).

\begin{proposition}
\label{prop:2202a1}
We have $q_J\equiv f$ and for each $j\in\{1,\ldots,J\}$,
\begin{align}
\label{eq:2408a1}
q_{j-1}(x)=&\EE\bigl[q_j(X_{\Delta, j\Delta})|X_{\Delta,(j-1)\Delta}=x\bigr]
=\frac{1}{2^{m}}\sum_{y=(y^{1},\ldots,y^{m})\in\left\{ -1,1\right\} ^{m}}q_j(\Phi_\Delta(x,y)).
\end{align}
Moreover, the coefficients~\eqref{eq:coef05} can be expressed in terms of the functions \(q_j,\) \(j=1,\ldots, J,\) as
\begin{align}
\label{eq:coef05a}
a_{j,r,s}(x)=\frac{1}{2^m}\sum_{y=(y^{1},\ldots,y^{m})\in\left\{ -1,1\right\} ^{m}}\,  \left [\prod_{i=1}^r y^{s_i}\right]
q_j(\Phi_\Delta(x,y))
\end{align}
for all $j$, $r$ and $s=(s_1,\ldots,s_r)$ as in~\eqref{eq:repr02}.
%An equivalent form of~\eqref{eq:coef06} is
%\begin{eqnarray}
%\label{eq:coef06a}
%a_{j,r}\bigl(x,(y^{i})_{i=1}^{r-1}\bigr)&=&\frac{1}{2^{m-r+1}}\sum_{(y^{r+1},\ldots,y^{m})\in\{-1,1\}^{m-r}}\{\Q_{1,j}(x,\overline{y})- \Q_{1,j}(x,\underline{y})\},
%\end{eqnarray}
%where 
%\begin{align}
%\overline{y}&\doteq(y^{1},\ldots,y^{r-1},1,y^{r+1},\ldots,y^{m}),\\
%\underline{y}&\doteq(y^{1},\ldots,y^{r-1},-1,y^{r+1},\ldots,y^{m})
%\end{align}
%for all $j\in\{1,\ldots,J\}$ and $r\in\{1,\ldots,m\}$.
\end{proposition}

\begin{discussion}
The advantage of the representation~\eqref{eq:coef05a}  over the original one consists in the fact that all functions \(q_j,\) \(j=1,\ldots, J,\) can be recursively computed using regressions over one time step
(based on the first equality in~\eqref{eq:2408a1}) and without involvement of the independent of $X_{\Delta,(j-1)\Delta}$
centred random variables $\xi_j^{s_i}$ (cf.~\eqref{eq:coef05}),
rendering the estimates for \(q_j\) more stable. If  \(q_j\) is approximated as a linear combination of \(n\) basis functions, then the cost of computing the coefficients in this linear combination by least squares regression on \(N\) paths is  of order \(N\times n^2.\) Once \(q_j\) is approximated, the cost of estimating \(a_{j,r,s}(x)\) in a given point \(x\) via \eqref{eq:coef05a} is of order
$2^m\times(c_1+c_2\times n)$,
where the constant $c_1$ describes the cost
of computing $\Phi_\Delta(x,y)$
for given points $x$ and~$y$
(this is $d\times m$ in case of~\eqref{eq:PhiK=1}),
and the constant $c_2$ describes the cost
of computing the value of a basis function
at a point in~$\bbR^d$
(this is typically~$d$).
\end{discussion}

\subsection{Second order schemes}
%\label{section:2:3}
Now we treat weak schemes of order~$2$.
We consider a scheme, where
$d$-dimensional approximations
$X_{\Delta,j\Delta}$, $j=0,\ldots,J$, satisfy
$X_{\Delta,0}=x_0$ and
\begin{align}
\label{eq:2002a5}
X_{\Delta,j\Delta}=
\Phi_{\Delta}(X_{\Delta,(j-1)\Delta},\xi_j,V_j),
\quad j=1,\ldots,J,
\end{align}
for some functions
$\Phi_{\Delta}\colon\bbR^{d+m+m\times m}\to\bbR^d$.
Here,
\begin{itemize}
\item[(S1)]
$\xi_j=(\xi_j^k)_{k=1}^m$
are $m$-dimensional random vectors,
\item[(S2)]
$V_j=(V_j^{kl})_{k,l=1}^m$
are random $m\times m$-matrices,
\item[(S3)]
the pairs $(\xi_j,V_j)$, $j=1,\ldots,J$, are i.i.d.,
\item[(S4)]
for each $j$, the random elements $\xi_j$ and $V_j$
are independent,
\item[(S5)]
for each $j$, the random variables
$\xi_j^k$, $k=1,\ldots,m$, are i.i.d.\ with
\begin{align*}
\PP\left(\xi_j^k=\pm\sqrt{3}\right)=\frac16,
\quad
\PP\left(\xi_j^k=0\right)=\frac23,
\end{align*}
\item[(S6)]
for each $j$, the random variables
$V_j^{kl}$, $1\le k<l\le m$, are i.i.d.\ with
\begin{align*}
\PP\left(V_j^{kl}=\pm1\right)=\frac12,
\end{align*}
\item[(S7)]
$V_j^{lk}=-V_j^{kl}$, $1\le k<l\le m$, $j=1,\ldots,J$,
\item[(S8)]
$V_j^{kk}=-1$, $k=1,\ldots,m$, $j=1,\ldots,J$.
\end{itemize}

\begin{remark}
%\label{rem:2002a1}
In order to obtain an order~2 weak scheme
in the multidimensional case,
we need to incorporate additional
random elements $V_j$
into the structure of the scheme.
This is the reason
why we now consider~\eqref{eq:2002a5}
instead of~\eqref{eq:scheme_structure_md}.
For instance, to get the
\emph{simplified order~2 weak Taylor scheme}
of \cite[Section~14.2]{KP}
in the multidimensional case,
we need to define the functions
$\Phi_{\Delta}(x,y,z)$,
$x\in\bbR^d$, $y\in\bbR^m$, $z\in\bbR^{m\times m}$,
as explained below.
First we define the function
$\Sigma\colon\bbR^d\to\bbR^{d\times d}$
by the formula
\begin{align*}
\Sigma(x)=\sigma(x)\sigma(x)^\top
\end{align*}
and recall that the coordinates
of vectors and matrices are denoted
by superscripts,
e.g.\ $\Sigma(x)=(\Sigma^{kl}(x))_{k,l=1}^d$,
$\Phi_{\Delta}(x,y,z)
=(\Phi_{\Delta}^k(x,y,z))_{k=1}^d$.
Let us introduce the operators
$\cL^r$, $r=0,\ldots,m$,
that act on sufficiently smooth functions
$g\colon\bbR^d\to\bbR$ as follows:
\begin{align*}
\cL^0 g(x)&\doteq\sum_{k=1}^d
\mu^k(x) \frac{\partial g}{\partial x^k}(x)
+\frac12 \sum_{k,l=1}^d
\Sigma^{kl}(x) \frac{\partial^2 g}{\partial x^l\partial x^k}(x),
%\label{eq:opL0}
\\
\cL^r g(x)&\doteq\sum_{k=1}^d \sigma^{kr}(x)
\frac{\partial g}{\partial x^k}(x),\quad
r=1,\ldots,m. %\label{eq:opLr}
\end{align*}
The $r$-th coordinate $\Phi_{\Delta}^r$,
$r=1,\ldots,d$, in the simplified order~2
weak Taylor scheme of \cite[Section~14.2]{KP}
is now given by the formula
\begin{align}
\Phi_{\Delta}^r(x,y,z)&=
x^r+\sum_{k=1}^m \sigma^{rk}(x)\,y^k\,\sqrt{\Delta}
\label{eq:2002a6}\\
&\hspace{1em}+\left[
\mu^r(x)+\frac12\sum_{k,l=1}^m
\cL^k\sigma^{rl}(x) (y^k y^l+z^{kl})
\right]\Delta
\notag\\
&\hspace{1em}+\frac12\sum_{k=1}^m
\left[
\cL^0\sigma^{rk}(x)+\cL^k \mu^r(x)
\right]
y^k\,\Delta^{3/2}
+\frac12\cL^0\mu^r(x)\,\Delta^2,
\notag
\end{align}
provided the coefficients $\mu$ and $\sigma$
of~\eqref{x_sde}
are sufficiently smooth.
We will need to work
explicitly with~\eqref{eq:2002a6}
at some point,
but all results in this subsection
assume structure~\eqref{eq:2002a5} only.
\end{remark}

Let us define the index sets
\begin{align*}
\cI_1=\{1,\ldots,m\},\quad
\cI_2=\left\{(k,l)\in\cI_1^2:k<l\right\}
\end{align*}
and the system
\begin{align*}
\cA=\left\{(U_1,U_2)\in\mathcal P(\cI_1)\times\mathcal
P(\cI_2):U_1\cup U_2\ne\emptyset\right\},
\end{align*}
where $\mathcal P(\cI)$ denotes
the set of all subsets of a set~$\cI$.
For any $U_1\subseteq\cI_1$
and $o\in\{1,2\}^{U_1}$,
we write $o$ as
$o=(o_r)_{r\in U_1}$.
Below we use the convention
that a product over the empty set
is always one.

\begin{theorem}
\label{th:weak_md03}
It holds
\begin{align}
\label{eq:2002a1}
f(X_{\Delta,T})=\EE f(X_{\Delta,T})
+\sum_{j=1}^J
\sum_{(U_1,U_2)\in\cA}
\sum_{o\in\{1,2\}^{U_1}}
a_{j,o,U_1,U_2}(X_{\Delta,(j-1)\Delta})
\prod_{r\in U_1} H_{o_r}(\xi_j^r)
\prod_{(k,l)\in U_2} V_j^{kl},
\end{align}
where the coefficients
$a_{j,o,U_1,U_2}\colon\bbR^d\to\bbR$
can be computed by the formula
\begin{align}
\label{eq:2002a2}
a_{j,o,U_1,U_2}(x)
=\EE\left[\left.
f(X_{\Delta,T})
\prod_{r\in U_1} H_{o_r}(\xi_j^r)
\prod_{(k,l)\in U_2} V_j^{kl}
\right| X_{\Delta,(j-1)\Delta}=x
\right].
\end{align}
\end{theorem}

Combining Theorem~\ref{th:weak_md03}
with Theorem~2.1 in~\cite{MilsteinTretyakov:2004}
we obtain the following result,
which provides a bound for the discretisation error
and a perfect control variate for the discretised quantity.

\begin{proposition}
%\label{prop:second:CV}
Assume, that $\mu$ and $\sigma$ in~\eqref{x_sde} are Lipschitz continuous with components
$\mu^i,\,\sigma^{i,r}\colon \R^d\to\R$,
$i=1,\ldots,d$, $r=1,\ldots,m$,
being $6$ times continuously differentiable
with their partial derivatives of order up to $6$
having polynomial growth.
Let $f\colon\R^d\to\R$ be $6$ times continuously differentiable
with partial derivatives of order up to $6$
having polynomial growth.
Provided that~\eqref{eq:2002a6} holds
and that, for sufficiently large $p\in\mathbb N$,
the expectations $\EE |X_{\Delta,j\Delta}|^{2p}$
are uniformly bounded in $J$ and $j=0,\ldots,J$,
we have for this
``simplified second order weak Taylor scheme''
\begin{align*}
\left|\E\left[f(X_T) - f(X_{\Delta,T})\right]\right|\le c\Delta^2,
\end{align*}
where the constant $c$ does not depend on $\Delta$. Moreover, we have
$\Var\left[f(X_{\Delta,T}) - M^{(2)}_{\Delta,T}\right]=0$
for the control variate
\begin{align}
\label{eq:28042016a1}
M^{(2)}_{\Delta,T}\doteq\sum_{j=1}^J
\sum_{(U_1,U_2)\in\cA}
\sum_{o\in\{1,2\}^{U_1}}
a_{j,o,U_1,U_2}(X_{\Delta,(j-1)\Delta})
\prod_{r\in U_1} H_{o_r}(\xi_j^r)
\prod_{(k,l)\in U_2} V_j^{kl},
\end{align}
where the coefficients
$a_{j,o,U_1,U_2}(x)$ are defined in~\eqref{eq:2002a2}.
\end{proposition}

\subsubsection{Computation of coefficients}
\label{subsubsec:321}
Similarly to the case of first order schemes, one can derive an alternative representation for the coefficients~\eqref{eq:2002a2} making their computation more stable.
The next result contains backward recursions
for the functions $q_j$ of~\eqref{tower:q} and for
$a_{j,o,U_1,U_2}$ of~\eqref{eq:2002a2}.

\begin{proposition}
\label{prop:0403a1}
We have $q_J\equiv f$ and, for each $j\in\{1,\ldots,J\}$,
\begin{align}
\label{eq:0403a12}
q_{j-1}(x)=&\EE[q_{j}(X_{\Delta,j\Delta})|X_{\Delta,(j-1)\Delta}=x]
\\
\notag
=&\frac{1}{2^{\frac{m(m-1)}{2}}}\,\frac{1}{6^{m}}\sum_{(y^{1},\ldots,y^{m})\in\{-\sqrt{3},0,\sqrt{3}\}^{m}}\sum_{(z^{uv})_{1\le u<v\le m}\in\{-1,1\}^{\frac{m(m-1)}{2}}}4^{\sum_{i=1}^{m}I(y^{i}=0)}q_j(\Phi_\Delta(x,y,z)),
\end{align}
and, for all $j\in\{1,\ldots,J\}$,
$(U_1,U_2)\in\cA$ and $o\in\{1,2\}^{U_1}$,
it holds
\begin{align}
\label{eq:0403a11}
a_{j,o,U_1,U_2}(x)=\frac{1}{2^{\frac{m(m-1)}{2}}}\,\frac{1}{6^{m}}\sum_{(y^{1},\ldots,y^{m})\in\{-\sqrt{3},0,\sqrt{3}\}^{m}}\sum_{(z^{uv})_{1\le u<v\le m}\in\{-1,1\}^{\frac{m(m-1)}{2}}}\\
\notag 4^{\sum_{i=1}^{m}I(y^{i}=0)}
\prod_{r\in U_1}H_{o_r}(y^r)
\prod_{(k,l)\in U_2}z^{kl}\,
q_j(\Phi_\Delta(x,y,z)),
\end{align}
where $y=(y^1,\ldots,y^m)$
and $z=(z^{uv})$ is the $m\times m$-matrix
with $z^{vu}=-z^{uv}$, $u<v$, $z^{uu}=-1$.
\end{proposition}

\section{Generic regression algorithm}
\label{sec:4}
In the previous sections
we have given several representations
for perfect control variates.
Now we discuss how to compute the coefficients
in these representations via regression.
For the sake of clarity,
we focus on second order schemes
and representation~\eqref{eq:2002a1}
with coefficients given by~\eqref{eq:2002a2}.

\subsection{Monte Carlo regression}
Fix a $n$-dimensional vector of real-valued functions \ensuremath{\psi=(\psi^{1},\ldots,\psi^{n})}
on \ensuremath{\mathbb{R}^{d}}. Simulate
a big number\footnote{In the complexity analysis below
we show how large $N$ is required to be
in order to provide an estimate within some
given tolerance.}
$N$ of independent  ``training paths'' of the
discretised diffusion $X_{\Delta,j\Delta},$ $j=0,\ldots, J$.
In what follows these $N$ training paths
are denoted by $D_N^{tr}$:
$$
D_N^{tr}\doteq
\left\{
(X_{\Delta,j\Delta}^{tr,(i)})_{j=0,\ldots,J}:
i=1,\ldots,N
\right\}.
$$
Let
$\boldsymbol{\alpha}_{j,o,U_1,U_2}=(
\alpha_{j,o,U_1,U_2}^{1},\ldots,\alpha_{j,o,U_1,U_2}^{n})$,
where $j\in\left\{1,\ldots, J\right\}$, $(U_1,U_2)\in\mathcal{A}$, $o\in\left\{1,2\right\}^{U_1}$,
 be a solution of the following least squares optimisation problem:
\begin{align*}
\operatorname{argmin}_{\boldsymbol{\alpha}\in\mathbb{R}^{n}}
\sum_{i=1}^{N}\left[\zeta^{tr,(i)}_{j,o,U_1,U_2}-\alpha^{1}\psi^{1}(X_{\Delta,(j-1)\Delta}^{tr,(i)})-\ldots-\alpha^{n}\psi^{n}(X_{\Delta,(j-1)\Delta}^{tr,(i)})\right]^{2}
\end{align*}
with 
\begin{align*}
\zeta_{j,o,U_1,U_2}^{tr,(i)}\doteq f(X_{\Delta,T}^{tr,(i)})
\prod_{r\in U_1} H_{o_r}\left((\xi_j^{tr,(i)})^r\right)
\prod_{(k,l)\in U_2} (V_j^{tr,(i)})^{kl}.
\end{align*}
Define an estimate for  the coefficient function $a_{j,o,U_1,U_2}$ via
\begin{align*}
\hat a_{j,o,U_1,U_2}(x)\doteq
\hat a_{j,o,U_1,U_2}(x,D_N^{tr})\doteq
\alpha_{j,o,U_1,U_2}^{1}\psi^{1}(x)+\ldots+\alpha_{j,o,U_1,U_2}^{n}\psi^{n}(x),\quad x\in\mathbb{R}^{d}.
\end{align*}
The intermediate expression
$\hat a_{j,o,U_1,U_2}(x,D_N^{tr})$
in the above formula
emphasises that the estimates
$\hat a_{j,o,U_1,U_2}$
of the functions $a_{j,o,U_1,U_2}$
are random in that they depend on
the simulated training paths.
The cost of computing
$\boldsymbol{\alpha}_{j,o,U_1,U_2}$ is of order $O(Nn^{2})$,
 since each \ensuremath{\boldsymbol{\alpha}_{j,o,U_1,U_2}}
 is of the form \ensuremath{\boldsymbol{\alpha}_{j,o,U_1,U_2}=B^{-1}b}
 with 
\begin{align}
\label{b_matr_reg}
B_{k,l}\doteq\frac{1}{N}\sum_{i=1}^{N}\psi^{k}\bigl(X_{\Delta,(j-1)\Delta}^{tr,(i)}\bigr)\psi^{l}\bigl(X_{\Delta,(j-1)\Delta}^{tr,(i)}\bigr)
\end{align}
and 
\begin{align*}
b_{k}\doteq\frac{1}{N}\sum_{i=1}^{N}\psi^{k}\bigl(X_{\Delta,(j-1)\Delta}^{tr,(i)}\bigr)\,\zeta_{j,o,U_1,U_2}^{tr,(i)},
\end{align*}
\ensuremath{k,l\in\{1,\ldots,n\}.} The cost of approximating the family of the coefficient functions $a_{j,o,U_1,U_2}$, $j\in\left\{1,\ldots, J\right\}$, $(U_1,U_2)\in\mathcal{A}$, $o\in\left\{1,2\right\}^{U_1}$, is of order
$O\bigl(J(3^m 2^{\frac{m(m-1)}{2}}-1)Nn^{2}\bigr)$.

\subsection{Summary of the algorithm}
The algorithm consists of two phases:
training phase and testing phase.
In the training phase, we simulate
$N$ independent training paths $D_N^{tr}$
and construct regression estimates
$\hat a_{j,o,U_1,U_2}(\cdot,D_N^{tr})$
for the coefficients $a_{j,o,U_1,U_2}(\cdot)$.
In the testing phase,
independently from $D_N^{tr}$
we simulate $N_0$ independent testing paths
$(X_{\Delta,j\Delta}^{(i)})_{j=0,\ldots,J}$,
$i=1,\ldots,N_0$,
and build the Monte Carlo estimator
for $\EE[f(X_T)]$ as
\begin{equation}
\label{eq:0110a2}
\mathcal E=
\frac1{N_0}
\sum_{i=1}^{N_0}
\left(f(X^{(i)}_{\Delta,T})-\widehat M^{(2),(i)}_{\Delta,T}\right),
\end{equation}
where
\begin{align}\label{eq:19102016b1}
\widehat{M}^{(2),(i)}_{\Delta,T}\doteq\sum_{j=1}^J
\sum_{(U_1,U_2)\in\cA}
\sum_{o\in\{1,2\}^{U_1}}
\hat{a}_{j,o,U_1,U_2}(X^{(i)}_{\Delta,(j-1)\Delta},D_N^{tr})
\prod_{r\in U_1} H_{o_r}(\xi_j^{r,(i)})
\prod_{(k,l)\in U_2} V_j^{kl,(i)}
\end{align}
(cf.\ with~\eqref{eq:28042016a1}).
Due to the martingale transform structure
in~\eqref{eq:19102016b1}
(recall footnote~\ref{ft:19102016a1}
on page~\pageref{ft:19102016a1}),
we have
$\EE\left[\widehat M^{(2),(i)}_{\Delta,T}|D^{tr}_N\right]=0$,
hence
$\EE[\mathcal E|D^{tr}_N]
=\EE[f(X^{(i)}_{\Delta,T})-\widehat M^{(2),(i)}_{\Delta,T}|D^{tr}_N]
=\EE[f(X_{\Delta,T})]$,
and we obtain
\begin{align*}
\Var[\mathcal E]
&=\EE[\Var(\mathcal E|D^{tr}_N)]
+\Var[\EE(\mathcal E|D^{tr}_N)]
=\EE[\Var(\mathcal E|D^{tr}_N)]\\
&=\frac1{N_0}
\EE\left[\Var\left(f(X^{(1)}_{\Delta,T})-\widehat M^{(2),(1)}_{\Delta,T}|D^{tr}_N\right)\right]
=\frac1{N_0}
\Var\left[f(X^{(1)}_{\Delta,T})-\widehat M^{(2),(1)}_{\Delta,T}\right].
\end{align*}
Summarising, we have
\begin{align}
\EE[\mathcal E]&=\EE[f(X_{\Delta,T})],
\label{eq:19102016b2}\\
\Var[\mathcal E]&=\frac1{N_0}
\Var\left[f(X^{(1)}_{\Delta,T})-\widehat M^{(2),(1)}_{\Delta,T}\right].
\label{eq:19102016b3}
\end{align}
Notice that the result of~\eqref{eq:19102016b3}
indeed requires the computations above
and cannot be stated right from the outset
because the summands in~\eqref{eq:0110a2}
are dependent (through~$D^{tr}_N$).

This concludes the description
of the generic regression algorithm
for constructing the control variate.
Further details,
such as bounds for the right-hand side
of~\eqref{eq:19102016b3},
depend on a particular implementation,
i.e.\ on the quality of the chosen basis functions.
In what follows, we perform a detailed analysis
for the specific choice of the basis functions,
which leads to the so-called
piecewise polynomial partitioning estimates.

\section{Error bounds for piecewise polynomial regression}
\label{sec:4.5}
We fix some $p\in\bbN_0$,
which will denote the maximal degree
of polynomials involved in our basis functions.
The piecewise polynomial partitioning estimate of $a_{j,o,U_1,U_2}$ works as follows: 
consider some $R>0$ and an equidistant partition of $\left[-R,R\right]^d$ in $Q^d$ cubes $K^1,\ldots,K^{Q^d}$. Further, consider the basis functions $\psi^{k,1},\ldots,\psi^{k,n}$ with $k\in\left\{1,\ldots,Q^d\right\}$ and $n=\binom{p+d}{d}$ such that $\psi^{k,1}(x),\ldots,\psi^{k,n}(x)$ are polynomials with degree less than or equal to $p$ for $x\in K^k$ and $\psi^{k,1}(x)=\ldots=\psi^{k,n}(x)=0$ for $x\notin K^k$. Then we obtain the least squares regression estimate $\hat a_{j,o,U_1,U_2}(x)$ for $x\in\mathbb{R}^d$ as described in Section~\ref{sec:4}, based on $Q^dn=O(Q^dp^d)$ basis functions. In particular, we have $\hat a_{j,o,U_1,U_2}(x)=0\) for any \(x\notin\left[-R,R\right]^d$. We note that the cost of computing $\hat a_{j,o,U_1,U_2}$ for all $j,o,U_1,U_2$ is of order $O(J NQ^{d} p^{2d})$ rather than $O(J N Q^{2d} p^{2d})$ due to a block diagonal matrix structure
of $B$ in~\eqref{b_matr_reg}. An equivalent approach, which leads to the same estimator $\hat a_{j,o,U_1,U_2}(x)$, is to perform separate regressions for each cube $K^1,\ldots,K^{Q^d}$. Here, the number of basis functions at each regression is of order $O(p^d)$ so that the overall cost is of order $O(J N Q^{d} p^{2d})$, too.
For $x=(x_1,\ldots,x_d)\in\bbR^d$ and $h\in[1,\infty)$,
we will use the notations
\begin{align*}
|x|_h\doteq\bigg(\sum_{i=1}^d |x_i|^h\bigg)^{1/h},\quad
|x|_\infty\doteq\max_{i=1,\ldots,d}|x_i|.
\end{align*}
For $s\in\bbN_0$,
$C>0$ and $h\in[1,\infty]$,
we say that
a function $f\colon\bbR^d\to\bbR$ is
\emph{${(s+1,C)}$-smooth w.r.t.\ the norm $\left|\cdot\right|_h$} whenever, for all
$\alpha=(\alpha_1,\ldots,\alpha_d)\in\bbN_0^d$
with $\sum_{i=1}^d \alpha_i=s$, we have
\begin{align*}
|\partial_\alpha f(x)-\partial_\alpha f(y)|\le C|x-y|_h,
\quad x,y\in\bbR^d,
\end{align*}
i.e.\ the function $\partial_\alpha f$
is globally Lipschitz
with the Lipschitz constant $C$
with respect to the norm $|\cdot|_h$
on $\bbR^d$
(cf.~Definition~3.3 in~\cite{gyorfi2002distribution}).
In what follows,
we use the notation
$\PP_{\Delta,j-1}$
for the distribution of $X_{\Delta,(j-1)\Delta}$. In particular, we will work with the corresponding $L^2$-norm:
\begin{align*}
\|g\|^2_{L^2(\P_{\Delta,j-1})}\doteq \int_{\mathbb{R}^d} g^2(x)\,\P_{\Delta,j-1}(dx)=\mathbb{E}\left[g^2\left(X_{\Delta,(j-1)\Delta}\right)\right].
\end{align*}
Let us now fix some $j\in\left\{1,\ldots, J\right\}$, $(U_1,U_2)\in\mathcal{A}$, $o\in\left\{1,2\right\}^{U_1}$, set
\begin{align*}
\zeta_{j,o,U_1,U_2}\doteq f(X_{\Delta,T})
\prod_{r\in U_1} H_{o_r}(\xi_j^r)
\prod_{(k,l)\in U_2} V_j^{kl}
%\\
%\cK_1\doteq\{r\in U_1:o_r=1\}, \quad
%\cK_2\doteq\{r\in U_1:o_r=2\}
%\label{eq:0403a9}
\end{align*}
%and denote by~$k_1$ (resp.~$k_2$) the cardinality of~$\cK_1$ (resp.~$\cK_2$).
%Note that $k_1$ and/or $k_2$ can be equal to zero. Furthermore, we denote by $s$ the cardinality of $U_2$ which can be equal to zero, too. However, it always holds $s+k_1+k_2>0$ due to $U_1\cup U_2\ne\emptyset$. 
and remark that
$a_{j,o,U_1,U_2}(x)=\EE[\zeta_{j,o,U_1,U_2}|X_{\Delta, (j-1)\Delta}=x
]$. We assume that, for some constant $h\in[1,\infty]$ and some positive
constants $\Sigma,A,C_h,\nu,B_\nu$, it holds:
\begin{itemize}
\item[(A1)]
$\sup_{x\in\R^d}\Var[\zeta_{j,o,U_1,U_2}|X_{\Delta,(j-1)\Delta}=x]
\le\Sigma^2<\infty$,
\item[(A2)]
$\sup_{x\in\R^d} |a_{j,o,U_1,U_2}(x)|\le A\,\Delta^{1/2}<\infty$,
\item[(A3)]
$a_{j,o,U_1,U_2}$ can be
extended to $\mathbb{R}^d$ in a $(p+1,C_h)$-smooth way
w.r.t.\ the norm $|\cdot|_h$,
%preserving the bound in~(A2),
\item[(A4)]
$\PP(|X_{\Delta,(j-1)\Delta}|_\infty>R)
\le B_\nu R^{-\nu}$ for all $R>0$.
\end{itemize}

\begin{remark}
Due to representation~\eqref{eq:0403a11}, the smoothness of the coefficients \(a_{j,o,U_1,U_2}\) is related to the smoothness of the one step conditional distribution of \(X_{\Delta,j\Delta}\), given
\({X_{\Delta,(j-1)\Delta}=x}\), for any \(j=1,\ldots,J\)
(recall the first equality in~\eqref{eq:0403a12}),
and to the smoothness in $x$
of the mapping \(\Phi_\Delta\) from~\eqref{eq:2002a5}.
In the case when the mapping \(\Phi_{\Delta}\)
is given by~\eqref{eq:2002a6}, its smoothness in \(x\)
is related to the smoothness of the coefficients
\(\mu\) and~\(\sigma\).
Let us also notice that it is only a matter of convenience
which $h$ to choose in~(A3)
because all norms $|\cdot|_h$ are equivalent.
\end{remark}

Let $\hat a_{j,o,U_1,U_2}$ be the piecewise
polynomial partitioning estimate
of $a_{j,o,U_1,U_2}$ described
in the beginning of this section.
By $\tilde a_{j,o,U_1,U_2}$ we denote the truncated estimate,
which is defined as follows: 
\begin{align}
\label{eq:30042016a1}
\tilde a_{j,o,U_1,U_2}(x)\doteq
T_{A\Delta^{1/2}}\hat a_{j,o,U_1,U_2}(x)
\doteq\begin{cases}
\hat a_{j,o,U_1,U_2}(x)&\text{if }
|\hat a_{j,o,U_1,U_2}(x)|\le A\,\Delta^{1/2},\\
A\,\Delta^{1/2}\sgn\hat a_{j,o,U_1,U_2}(x)
&\text{otherwise.}
\end{cases}
\end{align}
We again emphasise that, in fact,
$\tilde a_{j,o,U_1,U_2}(x)=\tilde a_{j,o,U_1,U_2}(x,D_N^{tr})$,
that is, the estimates
$\tilde a_{j,o,U_1,U_2}$
of the functions $a_{j,o,U_1,U_2}$
depend on the simulated training paths.

\begin{theorem}
\label{th:2104a1}
Under (A1)--(A4), we have
\begin{align}
\label{eq:2104a2}
\EE\|\tilde a_{j,o,U_1,U_2}-a_{j,o,U_1,U_2}\|^2_{L^2(\PP_{\Delta,j-1})}
&\le
\tilde c\left(\Sigma^2+A^2\,\Delta(\log N+1)\right)\frac{\binom{p+d}d Q^d}{N}
\\
\notag
&\hspace{1em}+\frac{8\,C_h^2}{(p+1)!^2 d^{2-2/h}}
\left(\frac{Rd}Q\right)^{2p+2}
+8A^2\,\Delta B_\nu R^{-\nu},
\end{align}
where $\tilde c$ is a universal constant.
\end{theorem}

It is worth noting that the expectation
in the left-hand side of~\eqref{eq:2104a2}
accounts for the averaging over the randomness
in $D_N^{tr}$. To explain this in more detail,
let $(X_{\Delta,j\Delta})_{j=0,\ldots,J}$
be a ``testing path'' which is independent
of the training paths $D_N^{tr}$. Then it holds
\begin{align*}
\|\tilde a_{j,o,U_1,U_2}-a_{j,o,U_1,U_2}\|^2_{L^2(\PP_{\Delta,j-1})}
&\equiv
\|\tilde a_{j,o,U_1,U_2}(\cdot,D_N^{tr})-a_{j,o,U_1,U_2}(\cdot)\|^2_{L^2(\PP_{\Delta,j-1})}\\
&=
\EE\left[
\left(\tilde a_{j,o,U_1,U_2}(X_{\Delta,(j-1)\Delta},D_N^{tr})
-a_{j,o,U_1,U_2}(X_{\Delta,(j-1)\Delta})\right)^2
\,|\,D_N^{tr}\right],
\end{align*}
hence,
\begin{align}\label{eq:30042016a4}
\EE\|\tilde a_{j,o,U_1,U_2}-a_{j,o,U_1,U_2}\|^2_{L^2(\PP_{\Delta,j-1})}
=\EE\left[
\left(\tilde a_{j,o,U_1,U_2}(X_{\Delta,(j-1)\Delta},D_N^{tr})
-a_{j,o,U_1,U_2}(X_{\Delta,(j-1)\Delta})\right)^2
\right],
\end{align}
which provides an alternative form for
the expression in the left-hand side
of~\eqref{eq:2104a2}.

We now estimate the variance of the random variable
\(f(X_{\Delta,T})-\widetilde{M}^{(2)}_{\Delta,T}\),
where
\begin{align}
\label{eq:cv2}
\widetilde{M}^{(2)}_{\Delta,T}\doteq\sum_{j=1}^J
\sum_{(U_1,U_2)\in\cA}
\sum_{o\in\{1,2\}^{U_1}}
\tilde{a}_{j,o,U_1,U_2}(X_{\Delta,(j-1)\Delta},D_N^{tr})
\prod_{r\in U_1} H_{o_r}(\xi_j^r)
\prod_{(k,l)\in U_2} V_j^{kl}.
\end{align}
Using the martingale transform structure
in~\eqref{eq:cv2} and~\eqref{eq:28042016a1}
(recall footnote~\ref{ft:19102016a1}
on page~\pageref{ft:19102016a1})
together with the orthonormality (in~$L^2$) of the system
$\prod_{r\in U_1} H_{o_r}(\xi_j^r)\prod_{(k,l)\in U_2} V_j^{kl}$,
we get by Theorem~\ref{th:2104a1}
\begin{align}
\Var[f(X_{\Delta,T})-\widetilde{M}^{(2)}_{\Delta,T}]
&=
\Var[M^{(2)}_{\Delta,T}-\widetilde{M}^{(2)}_{\Delta,T}]
\label{eq:0110a1}\\
&=\sum_{j=1}^J
\sum_{(U_1,U_2)\in\cA}
\sum_{o\in\{1,2\}^{U_1}}
\EE\|\tilde a_{j,o,U_1,U_2}-a_{j,o,U_1,U_2}\|^2_{L^2(\PP_{\Delta,j-1})}
\notag\\
&\le
J\left(3^m2^{\frac{m(m-1)}{2}}-1\right)\left\{\tilde c\left(\Sigma^2+A^2\,\Delta(\log N+1)\right)\frac{\binom{p+d}d Q^d}{N}\right.
\notag\\
&\hspace{1em}+\left.\frac{8\,C_h^2}{(p+1)!^2 d^{2-2/h}}
\left(\frac{Rd}Q\right)^{2p+2}
+8A^2\,\Delta B_\nu R^{-\nu}\right\}.
\notag
\end{align}
In the case of piecewise polynomial regression,
the estimator $\mathcal E$ given in~\eqref{eq:0110a2}
with ``hat'' replaced by ``tilde''
is an unbiased estimator of $\EE[f(X_{\Delta,T})]$,
and, by~\eqref{eq:19102016b3},
the upper bound for its variance
is $\frac1{N_0}$ times the last expression
in~\eqref{eq:0110a1}.

\section{Complexity analysis for piecewise polynomial regression}
\label{sec:5}
Below we present a complexity analysis,
which explains how we can go beyond
the complexity order $\varepsilon^{-2}$
with $\varepsilon$ being the precision
to be achieved.\footnote{Notice that the multilevel Monte Carlo
(MLMC)
algorithm can at best achieve the complexity of order $\varepsilon^{-2}$.}

%The formulas for the optimal choice
%of the parameters $J$, $R$, $Q$, $N$,
%and for the number of testing paths $N_0$
%for a fixed $\nu$ look too heavy.
%Therefore, we present somewhat easier
%``limiting'' formulas as $\nu\to\infty$
%assuming that the constant $B_\nu$
%in~(A4) is of the order
%$$
%B_\nu\asymp Kb^\nu
%\quad\text{as }\nu\to\infty
%\quad(K>0,\;b>0).
%$$
%Loosely speaking this assumption accounts for the case
%when the solution $X$ of SDE~\eqref{x_sde}
%lives in the compact set $[-b,b]^d$.

%More interestingly, 
We will consider two variants of the
Monte Carlo approach with regression-based control variate.
The first algorithm, which is abbreviated below as
\emph{RCV approach}
(``RCV'' stands for ``Regression-based Control Variate''),
is the algorithm described in detail in Section~\ref{sec:4}.
Here the estimates $\tilde a_{j,o,U_1,U_2}$
needed in~\eqref{eq:cv2} are constructed via regressions
based on~\eqref{eq:2002a2}.
In the second algorithm,
which we call
\emph{recursive RCV (RRCV) approach},
we construct in the training phase regression-based estimates
$\tilde q_j$ of the functions $q_j$
backwards in time via regressions
based on the first equality in~\eqref{eq:0403a12}.
Given the approximations $\tilde q_j(\cdot,D_N^{tr})$
of the functions $q_j(\cdot)$,
we construct in the testing phase the approximations
of the values
$\tilde a_{j,o,U_1,U_2}(X_{\Delta,(j-1)\Delta}^{(i)},D_N^{tr})$
on the testing paths via~\eqref{eq:0403a11}
with $q_j(\cdot)$ replaced by $\tilde q_j(\cdot,D_N^{tr})$.
Then, again, the values of the control variate
on the testing paths are computed via~\eqref{eq:cv2},
and the Monte Carlo estimator for $\EE f(X_T)$
is computed as in~\eqref{eq:0110a2}.

%In what follows it is convenient to have notations
%for the following constants:
%$$
%c_{p,d}\doteq\binom{p+d}{d},\quad
%c_m\doteq3^m2^{\frac{m(m-1)}2}.
%$$

\subsection{Complexity analysis of the RCV approach}
\label{sec:direct}
%We perform $(c_m-1)$ regressions in the training phase
%and $(c_m-1)$ evaluations of $\tilde a_{j,o,U_1,U_2}$
%in the testing phase (using the regression coefficients
%from the training phase) at each time step.
%Therefore, 
The overall cost of the algorithm (training and testing phase) is of order
\begin{align}
\label{cost_rcv}
\cC\asymp JQ^d\max\left\{N,N_0\right\}, 
\end{align}
provided that we only track the parameters $J,N,N_0,Q$ that tend to infinity when $\epsilon\searrow 0$.
Further, we have the following constraints
\begin{align}
\label{constr_rcv}
\max\left\{\frac{1}{J^4},
\frac{JQ^d}{NN_0},
\frac{%C_h^2
J}{%d^{2-\frac{2}{h}}
N_0}\left(\frac{R}{Q}\right)^{2(p+1)},
\frac{%K
1}{R^\nu N_0}\right\}\lesssim\varepsilon^2,
\end{align}
provided that we, in addition to $J,N,N_0,Q$, track the parameter $R$, which also tends to infinity when $\epsilon\searrow 0$.
%as well as the constant $B_{\nu}$, which depends on the parameter $\nu$. 
Note that the first term in~\eqref{constr_rcv} comes from the squared bias of the estimator and the remaining three ones come from the variance of the estimator (see~\eqref{eq:0110a1} and~\eqref{eq:0110a2}).
%It is natural to expect that the optimal solution is given by
%all constraints being active as well as $N\asymp N_0$
%(the latter means that the costs of the training
%and testing phases are of the same order).

\begin{theorem}
\label{compl_rcv}
%In the limiting case $\nu\to\infty$ 
We obtain the following parameter values
\begin{align}
\label{sol_rcv}
J\asymp \varepsilon^{-\frac{1}{2}},\quad Q\asymp \varepsilon^{-\frac{5\nu+6(p+1)}{2d\nu+4(p+1)(2\nu+d)}}, \quad R\asymp \varepsilon^{-\frac{6(p+1)-d}{2d\nu+4(p+1)(2\nu+d)}},\quad N\asymp N_0\asymp \varepsilon^{-\frac{5d\nu+2(p+1)(5\nu+4d)}{2d\nu+4(p+1)(2\nu+d)}}, 
\end{align}
provided that
$p>\frac{d-2}{2}$ and $\nu>\frac{2d(p+1)}{2(p+1)-d}$.\footnote{\label{fn_rcv}When deriving
the solution via Lagrange multipliers (cf. proof of Theorem~\ref{compl_rcv})
one can see that these parameter values are
\emph{not} optimal if $p\le\frac{d-2}{2}$
or $\nu\le\frac{2d(p+1)}{2(p+1)-d}$ 
(a Lagrange multiplier corresponding to
a ``$\le0$'' constraint is negative).
Therefore, the recommendation is to choose
$p\in\bbN_0$ and $\nu>0$
according to $p>\frac{d-2}2$
and $\nu>\frac{2d(p+1)}{2(p+1)-d}$.
The opposite choice is allowed as well
(the method converges),
but theoretical complexity of the method
would be then worse than that of the SMC.}
As a result the complexity order is given by
\begin{align}
\label{eq:compld}
\mathcal{C}_{RCV}\asymp JQ^dN\asymp JQ^dN_0\asymp \varepsilon^{-\frac{11d\nu+2(p+1)(7\nu+8d)}{2d\nu+4(p+1)(2\nu+d)}}.
\end{align}
\end{theorem}

\subsection{Complexity of the RRCV approach}
\label{sec:towering}
In the training phase, the cost of approximating
all functions \(q_j\) is of order $N J Q^d$.
In the testing phase, the coefficients $\tilde a_{j,o,U_1,U_2}$
are computed via direct summation in~\eqref{eq:0403a11}
(with $q_j$ replaced by their approximations~$\tilde q_j$)
at a cost of order
$N_0 J Q^d
% + N_0 J
$,
and, finally, the control variate is computed via~\eqref{eq:cv2}
on all testing paths at a cost of order $N_0 J$.
Therefore, the overall cost is of order
$JQ^d\max\left\{N,N_0\right\}$, which is the same as for the RCV approach.
(In the latter formula we ignore
the cost constituents of smaller orders.)

We now establish the constraints
that are pertinent to the RRCV approach.
The regressions are now performed for the functions~$q_j$.
Pertinent assumptions are in the spirit of (A1)--(A4)
with different bounds in~(A1) and~(A2):
the conditional variance in such regressions
over one time step is typically of order $\Delta$,
hence we require the bound $\Sigma^2\Delta$
in the analogue of~(A1);
while in the analogue of~(A2)
and in formula~\eqref{eq:30042016a1}
for the truncated estimate
we require only the constant bound~$A$.
For the regression error,
instead of~\eqref{eq:2104a2} we get
\begin{align}
\label{eq:30042016a2}
\EE\|\tilde q_j-q_j\|^2_{L^2(\PP_{\Delta,j})}
&\le
\tilde c\left(\Sigma^2\Delta+A^2(\log N+1)\right)\frac{\binom{p+d}d Q^d}{N}
\\
\notag
&\hspace{1em}+\frac{8\,C_h^2}{(p+1)!^2 d^{2-2/h}}
\left(\frac{Rd}Q\right)^{2p+2}
+8A^2B_\nu R^{-\nu}.
\end{align}
It turns out that
\begin{align}\label{eq:30042016a3}
\EE\|\tilde a_{j,o,U_1,U_2}-a_{j,o,U_1,U_2}\|^2_{L^2(\PP_{\Delta,j-1})}
\le\EE\|\tilde q_j-q_j\|^2_{L^2(\PP_{\Delta,j})},
\end{align}
for all $j$, $o$, $U_1$ and~$U_2$.
To prove~\eqref{eq:30042016a3},
we use~\eqref{eq:30042016a4}
and the similar formula involving
$q_j$ and~$\tilde q_j$.
As in~\eqref{eq:30042016a4},
we consider a testing path
$(X_{\Delta,j\Delta})_{j=0,\ldots,J}$
which is independent of $D_N^{tr}$.
Since $\tilde a_{j,o,U_1,U_2}(\cdot,D_N^{tr})$
is given by~\eqref{eq:0403a11}
with $q_j(\cdot)$ replaced by $\tilde q_j(\cdot,D_N^{tr})$,
it holds
\begin{align*}
\tilde a_{j,o,U_1,U_2}(X_{\Delta,(j-1)\Delta},D_N^{tr})
=\EE\left[\left.
\tilde q_j(X_{\Delta,j\Delta},D_N^{tr})
\prod_{r\in U_1} H_{o_r}(\xi_j^r)
\prod_{(k,l)\in U_2} V_j^{kl}
\right|
X_{\Delta,(j-1)\Delta},D_N^{tr}
\right].
\end{align*}
Furthermore, we have
\begin{align*}
a_{j,o,U_1,U_2}(X_{\Delta,(j-1)\Delta})
=\EE\left[\left.
q_j(X_{\Delta,j\Delta})
\prod_{r\in U_1} H_{o_r}(\xi_j^r)
\prod_{(k,l)\in U_2} V_j^{kl}
\right|
X_{\Delta,(j-1)\Delta},D_N^{tr}
\right].
\end{align*}
The latter formula remains true also without conditioning
on $D_N^{tr}$, but this (seemingly superfluous)
conditioning is helpful in the following calculation:
\begin{align}
&\left(\tilde a_{j,o,U_1,U_2}(X_{\Delta,(j-1)\Delta},D_N^{tr})
-a_{j,o,U_1,U_2}(X_{\Delta,(j-1)\Delta})\right)^2
\label{eq:30042016a7}\\
&\hspace{3em}\le
\EE\left[\left.
\left(
\tilde q_j(X_{\Delta,j\Delta},D_N^{tr})
-q_j(X_{\Delta,j\Delta})
\right)^2
\right|
X_{\Delta,(j-1)\Delta},D_N^{tr}
\right]
\notag\\
&\hspace{7em}\times
\EE\left[\left.
\left(
\prod_{r\in U_1} H_{o_r}(\xi_j^r)
\prod_{(k,l)\in U_2} V_j^{kl}
\right)^2
\right|
X_{\Delta,(j-1)\Delta},D_N^{tr}
\right]
\notag\\
&\hspace{3em}=
\EE\left[\left.
\left(
\tilde q_j(X_{\Delta,j\Delta},D_N^{tr})
-q_j(X_{\Delta,j\Delta})
\right)^2
\right|
X_{\Delta,(j-1)\Delta},D_N^{tr}
\right].
\notag
\end{align}
We arrive at~\eqref{eq:30042016a3}
by taking expectations in~\eqref{eq:30042016a7}
and using~\eqref{eq:30042016a4}
together with the similar formula for
$q_j$ and~$\tilde q_j$.
Finally, we get an upper bound for the variance
in the RRCV approach
by the same calculation as in~\eqref{eq:0110a1}
using~\eqref{eq:30042016a2}
and~\eqref{eq:30042016a3}
(instead of~\eqref{eq:2104a2}),
and the resulting upper bound
is the same as in~\eqref{eq:0110a1}
except for that
$A^2\Delta$ is replaced by $A^2$,
while $\Sigma^2$ is replaced by $\Sigma^2\Delta$.
Thus, in the case of the RRCV approach,
our constraints are 
\begin{align}
\label{constr_rrcv}
\max\left\{\frac{1}{J^4},
\frac{JQ^d\log N}{NN_0},
\frac{%C_h^2
J}{%d^{2-\frac{2}{h}}
N_0}\left(\frac{R}{Q}\right)^{2(p+1)},
\frac{%K
J}{R^\nu N_0}\right\}\lesssim\varepsilon^2,
\end{align}
where we again only track the parameters $J,N,N_0,Q,R$. 
%It is natural to expect that the optimal solution is given by all constraints being active  as well as $c_{p,d}N\asymp c_m N_0$
%(so that the costs of the training and testing phases
%are of the same order). In the limiting case $\nu\to\infty$
%
\begin{theorem}
\label{compl_rrcv}
We obtain the following parameter values
\begin{align*}
&J\asymp \varepsilon^{-\frac{1}{2}},\quad Q\asymp \varepsilon^{-\frac{5\nu+10(p+1)}{2d\nu+4(p+1)(2\nu+d)}}, \quad R\asymp \varepsilon^{-\frac{5(p+1)}{2d\nu+4(p+1)(2\nu+d)}},\\
&N\asymp N_0\asymp \varepsilon^{-\frac{5d\nu+10(p+1)(\nu+d)}{2d\nu+4(p+1)(2\nu+d)}}
\sqrt{|\log\varepsilon|},
%\sqrt{\log\left(\varepsilon^{-\frac{5d\nu+10(p+1)(\nu+d)}{2d\nu+4(p+1)(2\nu+d)}}\right)},
\end{align*}
provided that $p>\frac{d-2}{2}$ and $\nu>\frac{2d(p+1)}{2(p+1)-d}$.\footnote{Footnote~\ref{fn_rcv} on page~\pageref{fn_rcv} applies.}
Thus, we have for the complexity 
\begin{align}
\label{eq:complt}
\mathcal{C}_{RRCV}\asymp JQ^dN\asymp JQ^dN_0\asymp \varepsilon^{-\frac{11d\nu+2(p+1)(7\nu+11d)}{2d\nu+4(p+1)(2\nu+d)}}
\sqrt{|\log\varepsilon|}.
%\sqrt{\log\left(\varepsilon^{-\frac{5d\nu+10(p+1)(\nu+d)}{2d\nu+4(p+1)(2\nu+d)}}\right)}.
\end{align}
\end{theorem}
%Notice that the $\log$-term in the formulas for $N$ and $N_0$ has been added afterwards to satisfy all constraints.

\subsection{Discussion}
\label{compl_disc}
For the sake of comparison with the SMC and MLMC approaches,
we recall at this point that their complexities are
$$
\mathcal C_{SMC}\asymp\varepsilon^{-2.5}
\quad\text{and}\quad
\mathcal C_{MLMC}\asymp\varepsilon^{-2}
$$
at best (we are considering the second order scheme).
Complexity estimates~\eqref{eq:compld} and~\eqref{eq:complt}
show that one can go beyond the complexity order
$\varepsilon^{-2}$, provided that
\begin{align*}
p>\frac{7d-2}{2},\quad \nu>\frac{8d(p+1)}{2(p+1)-7d}
\end{align*}
in case of the RCV approach and
\begin{align*}
p>\frac{7d-2}{2},\quad \nu>\frac{14d(p+1)}{2(p+1)-7d}
\end{align*}
in case of the RRCV approach.
Both in~\eqref{eq:compld} and~\eqref{eq:complt}
the power of $\varepsilon$ converges to $-1.75$ as $p,\nu\to\infty$ (the log-term is ignored).
Notice that, while $d$ and $m$ are fixed,
$p$ and $\nu$ are free parameters in our algorithms,
which can be chosen large,
provided the smoothness in $\mu$, $\sigma$ and $f$
allows that.
Therefore, whenever it is possible to take arbitrarily large
$p$ and $\nu$,
the complexity of our scheme can be reduced to
$\varepsilon^{-1.75-\delta}$
for arbitrarily small $\delta>0$.

Notice that we obtain such a complexity
for piecewise polynomial regression
with the second order weak scheme.
A natural question is to perform a similar
complexity analysis also for the weak Euler scheme.
We then get the complexity $\eps^{-2.5}$
in the limit as $p,\nu\to\infty$, that is,
both the RCV and the RRCV approaches
with the weak Euler scheme cannot outperform
the MLMC approach as well as the SMC approach
with the second order scheme
(but they still outperform the SMC approach
with the Euler or the weak Euler scheme
because the complexity of the latter is $\eps^{-3}$).
Still, both the RCV and the RRCV approaches
might be useful also with the weak Euler scheme,
provided we choose basis functions
other than those in piecewise polynomial regression
(recall the last paragraph in Section~\ref{sec:4}).

Obviously, the complexity estimate~\eqref{eq:compld} of the RCV approach gives us a better order compared to the one of the RRCV approach~\eqref{eq:complt} (due to the factor $J$ which arises in the last expression of the maximum term~\eqref{constr_rrcv} but not in~\eqref{constr_rcv}). However, the larger is the parameter $\nu$, the closer are both complexities to each other
(provided that we ignore the log-term).
As we mentioned in Sections~\ref{subsubsec:311} and~\ref{subsubsec:321}, from the computational point
of view it is preferable to consider the RRCV approach rather than the RCV one, since we perform regressions over only one time step in RRCV.
In addition, in case of the RCV approach, there are destabilising factors $\prod_{r\in U_1} H_{o_r}(\xi_j^r)\prod_{(k,l)\in U_2} V_j^{kl}$ in the estimation of $a_{j,o,U_1,U_2}$, which are independent of $X_{\Delta,(j-1)\Delta}$ and have zero expectation and thus may lead to poor regression results. Regarding the RRCV approach,
such destabilising factors are not present in
the regression for~$q_{j}$.

\section{Numerical results}
\label{sec:6}
In this section, we consider weak schemes of second order
and compare the numerical performance
of the SMC, MLMC, RCV and RRCV approaches.
For simplicity we implemented a global regression
(i.e.\ the one without truncation and partitioning, as a part of the general description in Section~\ref{sec:4}). In what follows it is convenient to have notations for the following constants
$$c_m:=3^m2^\frac{m(m-1)}{2},\quad c_{p,d}:=\binom{p+d}{d}+1.$$
Regarding the choice of basis functions,
we use in both RCV and RRCV approaches
the same polynomials $\psi(x)=\prod_{i=1}^dx_i^{l_i}$,
where $l_1,\ldots l_d\in\left\{0,1,\ldots,p\right\}$
and $\sum_{l=1}^dl_i\leq p$.
In addition to the polynomials,
we consider the function $f$ as a basis function.
Hence, we have overall $c_{p,d}$
basis functions in each regression. \par
The following results are based on program codes written and vectorised in MATLAB and running on a Linux 64-bit operating system.

\subsection{One-dimensional example}
\label{num_1d}
Here $d=m=1$.
We consider the following SDE
\begin{align}
\label{eq:sde1}
dX_t=&-\frac{1}{2}\tanh\left(X_t\right)\sech^2\left(X_t\right)dt+
\sech\left(X_t\right)dW_t,\quad X_0=0,
\end{align}
for $t\in\left[0,1\right]$, where $\sech(x)\doteq\frac{1}{\cosh(x)}$. This SDE has an exact solution \(X_t=\arsinh\left(W_t\right).\)
Furthermore, we consider the functional
$f(x)=\sech(x)+15\arctan(x)$, that is, we have
\begin{align}
\label{eq:0110a3}
\EE\left[f\left(X_1\right)\right]=
\EE\left[\sech\left(\arsinh\left(W_1\right)\right)\right]=
\EE\left[\frac{1}{\sqrt{1+W_1^2}}\right]\approx 0.789640.
\end{align}
We choose $p=3$ (that is, $5$~basis functions) and,
for each $\varepsilon=2^{-i}$, $i\in\left\{2,3,4,5,6\right\}$,
we set the parameters $J$, $N$ and \(N_0\) as follows
(compare with the formulas in Section~\ref{sec:5} for the ``limiting'' case $\nu\to\infty$ and ignore 
%the constant $B_\nu$ as well as 
the log-terms for the RRCV approach):
\begin{align*}
J=\left\lceil \varepsilon^{-0.5}\right\rceil,\quad N=c_N\cdot\lceil \varepsilon^{-1.3235}\rceil,\quad \quad c_{N}=\left\{\begin{array}{ll}64 & \text{RRCV} \\ 32 & \text{RCV}\end{array}\right.,\quad
N_0=128\cdot\lceil \varepsilon^{-1.3235}\rceil.
\end{align*}
Regarding the SMC approach, the number of paths is set $N_0=32\cdot \varepsilon^{-2}$. The factors 32, 64 and 128 are here for stability purposes. We use different constants for the training and testing paths due the fact that, if we also track the constants $c_{p,d}$ and $c_m$, we will have the cost of order $O(Jc_{p,d}(c_m-1)\max\left\{Nc_{p,d},N_0\right\})$ for the RCV approach and $O(Jc_{p,d}\max\left\{Nc_{p,d},N_0c_m\right\})$ for the RRCV approach (cf.~\eqref{cost_rcv}). Since we get from Theorems~\ref{compl_rcv} and~\ref{compl_rrcv} that both components in the maximum term are of the same order in the optimal solution, we choose the constants such that $Nc_{p,d}\approx N_0$ in case of the RCV approach and $Nc_{p,d}\approx N_0c_m$ in case of the RRCV approach.
%, such as compensating the effect of $B_\nu$.
As for the MLMC approach, we set the initial number of paths for the first level ($l=0$) equal to $10^3$ as well as the ``discretisation parameter'' $M=4$ (leading to timesteps of size $\frac{1}{4^l}$ at level $l$).
Next we compute the numerical RMSE
(the exact value is known, see~\eqref{eq:0110a3}) by means of \(100\) independent repetitions of the algorithm.
As can be seen from the left-hand side of Figure~\ref{compld},
the estimated numerical complexity is about $\text{RMSE}^{-1.41}$ for the RRCV approach, $\text{RMSE}^{-1.66}$ for the RCV approach, $\text{RMSE}^{-1.99}$ for the MLMC approach
and $\text{RMSE}^{-2.53}$ for the SMC approach,
which we get by regressing the log-time (logarithmic computing time of the whole algorithm in seconds) vs.\ log-RMSE.  
%Note that the MSE in the RRCV approach is dominated by the squared bias rather than by the variance. 
Thus, the 
%variance
complexity reduction works best
with the RRCV approach.

\begin{figure}[htb!]
\includegraphics[width=0.49\textwidth]{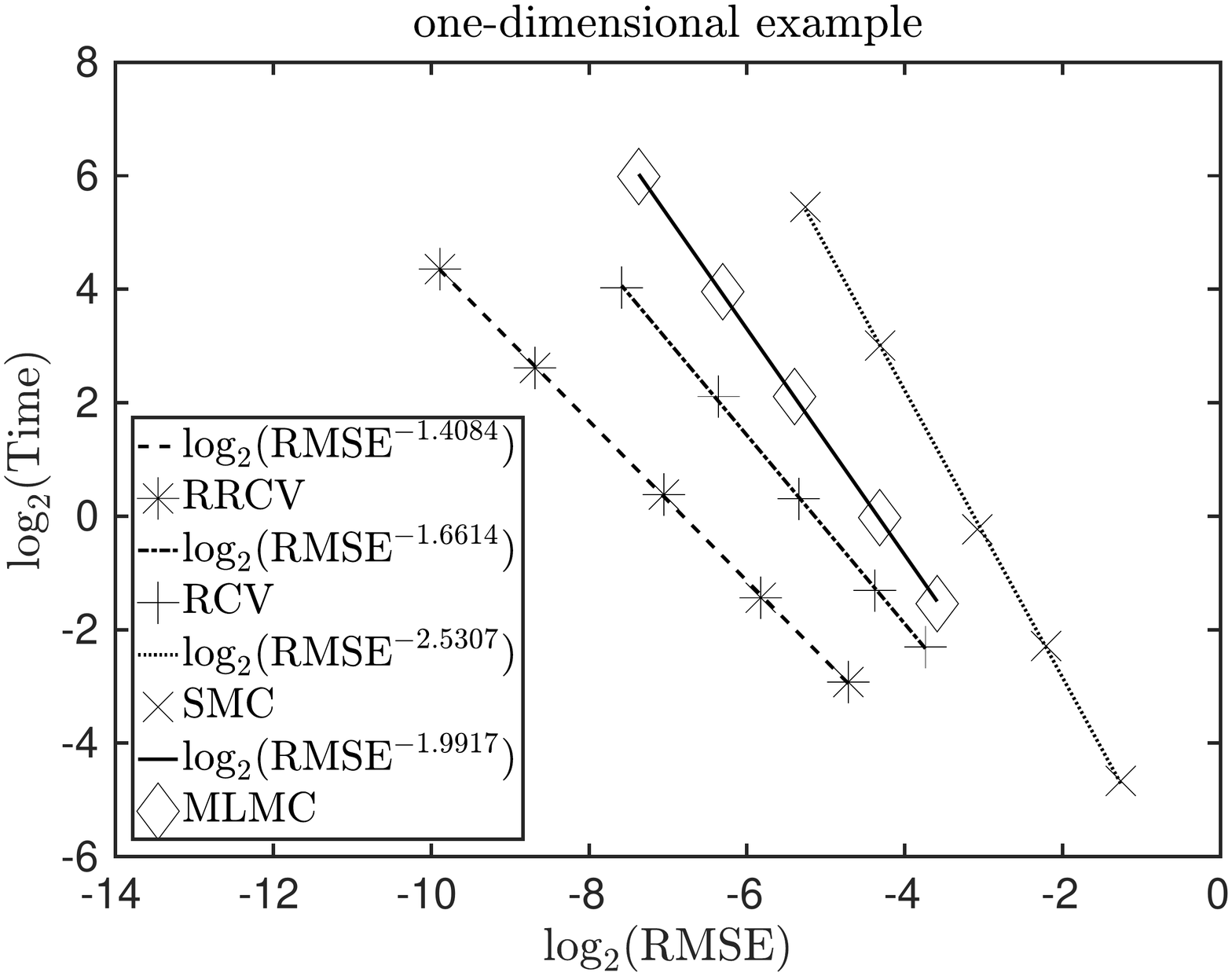}
\includegraphics[width=0.49\textwidth]{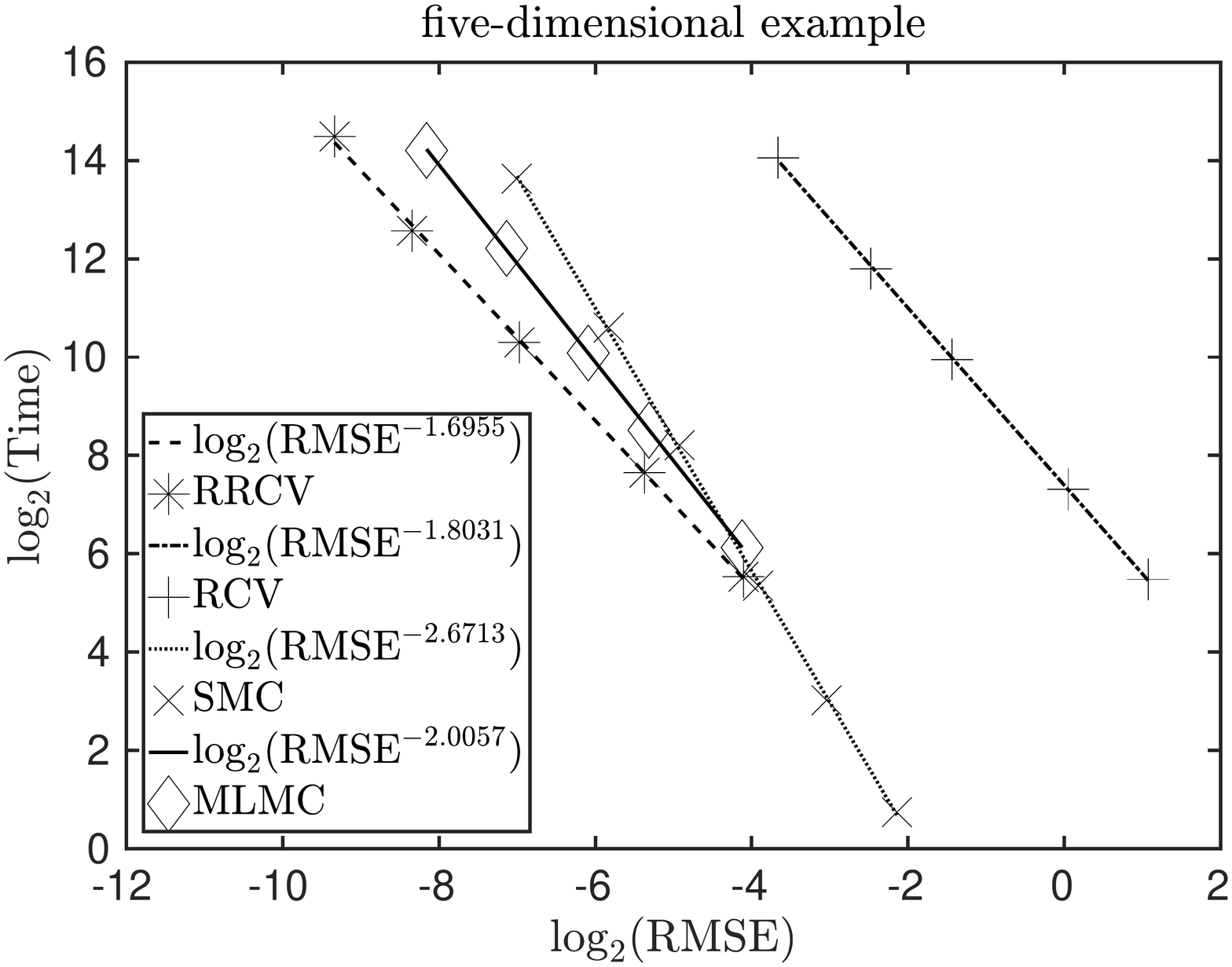}
\caption{Numerical complexities of the RRCV, RCV, SMC and MLMC approaches in the one- and five-dimensional case.}
\label{compld}
\end{figure}

\subsection{Five-dimensional example}
Here $d=m=5$.
We consider the SDE
\begin{align}
\notag
dX_t^i&=-\sin\left(X_t^i\right)\cos^3\left(X_t^i\right)dt+\cos^2\left(X_t^i\right)dW_t^i,\quad X_0^i=0,\quad i\in\left\{1,2,3,4\right\},\\
\label{5d_sde}
dX_t^5&=\sum_{i=1}^4\left[-\frac{1}{2}\sin\left(X_t^i\right)\cos^2\left(X_t^i\right)dt+\cos\left(X_t^i\right)
dW_t^i\right]+dW_t^5,\quad X_0^5=0.
\end{align}
The solution of~\eqref{5d_sde} is given by
\begin{align*}
X_t^i&=\arctan\left(W_t^i\right),\quad i\in\left\{1,2,3,4\right\},\\
X_t^5&=\sum_{i=1}^4\operatorname{arsinh}\left(W_t^i\right)
+W_t^5.
\end{align*}
for $t\in\left[0,1\right]$. Further, we consider the functional 
\begin{align*}
f(x)=\cos\left(\sum_{i=1}^5x^i\right)-20\sum_{i=1}^4\sin\left(x^i\right),
\end{align*}
that is, we have
\begin{align*}
\mathbb{E}\left[f\left(X_1\right)\right]&=\left(\mathbb{E}\left[\cos\left(
\arctan\left(W_1^1\right)+\operatorname{arsinh}\left(W_1^1\right)\right)\right]\right)^4
\mathbb{E}\left[\cos\left(W_1^5\right)\right]
%&=\left(\mathbb{E}\left[\cos\left(
%\arctan\left(W_1^1\right)+\operatorname{arsinh}\left(W_1^1\right)\right)\right]\right)^4
%\exp\left(-0.5\right)
\approx 0.002069.
\end{align*}
Note that we do not need to consider random variables $V_j^{kl}$
in the second order weak scheme,
since $\cL^k\sigma^{rl}(x)=0$ for $k\ne l$ (see~\eqref{eq:2002a6}).
This gives us a smaller constant $\tilde{c}_m:=3^m=243$
compared to $c_m=248832$ and hence a smaller number of terms for the control variate
(the factor $2^{\frac{m(m-1)}2}\equiv 1024$ is no longer present).
%This gives us a smaller number of terms for the control variate, namely a reduction of the factor $2^\frac{m(m-1)}{2}=1024$.
We again choose $p=3$
(this now results in $57$ basis functions),
consider the same values of $\varepsilon$ as above (and, in addition, consider the value $\varepsilon=2^{-7}$ for the SMC approach to obtain a similar computing time as for the RCV, RRCV and MLMC approaches). Moreover, we set
\begin{align*}
J=\left\lceil \varepsilon^{-0.5}\right\rceil,\quad N=c_N\cdot\lceil \varepsilon^{-1.5476}\rceil,\quad c_N=\left\{\begin{array}{ll}512 & \text{RRCV} \\ 32 & \text{RCV}\end{array}\right.,\\
N_0=c_{N_0}\cdot\lceil \varepsilon^{-1.5476}\rceil,\quad c_{N_0}=\left\{\begin{array}{ll}128 & \text{RRCV} \\ 1024 & \text{RCV}\end{array}\right.
\end{align*}
(similar to the previous example we consider the limiting case $\nu\to\infty$, ignore 
%the constant $B_\nu$ as well as 
the log-terms for the RRCV approach and consider the relations $Nc_{p,d}\approx N_0$ in case of the RCV approach and $Nc_{p,d}\approx N_0\tilde c_m$ in case of the RRCV approach). The number of paths for the SMC approach is set $N_0=512\cdot \varepsilon^{-2}$. Since the estimated variance of $f(X_{\Delta,T})$ is much higher than in the previous example, we use a higher constant here for the SMC approach. This is due to the fact that we get $N_0\gtrsim \Var\left[f(X_{\Delta,T})\right]\varepsilon^{-2}$ from the condition $
\Var\left[\frac{1}{N_0}\sum_{i=1}^{N_0}f(X_{\Delta,T}^{(i)})\right]=\frac{\Var\left[f(X_{\Delta,T})\right]}{N_0}\lesssim\varepsilon^2$. 
Regarding the MLMC approach, we again choose $M=4$, but the initial number of paths in the first level is increased to $10^4$.
As in the one-dimensional case, we compute the numerical RMSE
by means of 100 independent repetitions of the algorithm.
Our empirical findings are illustrated in the right-hand side of
Figure~\ref{compld}.
We observe the numerical complexities $\text{RMSE}^{-1.70}$ for the RRCV approach, $\text{RMSE}^{-1.80}$ for the RCV approach, $\text{RMSE}^{-2.01}$ for the MLMC approach
and $\text{RMSE}^{-2.67}$ for the SMC approach.
Even though here the complexity order of the RCV approach
is better than those of the MLMC and SMC approaches,
the RCV approach is practically outperformed
by the other approaches
(see Figure~\ref{compld}; the multiplicative constant influencing
the computing time is obviously very big).
However, the RRCV approach remains numerically the best one
also in this five-dimensional example.

%\begin{remark}[Right plots in
%Figures~\ref{onedim_compld} and~\ref{threedim_compld}]\label{rem:29042016a1}
%Both our approaches are variance reduction methods,
%while the bias is the same as in the SMC approach
%and is simply adopted from the discretisation scheme.
%Thus, the plots on the right-hand sides
%of Figures~\ref{onedim_compld} and~\ref{threedim_compld}
%are presented in order to provide a ``bias-free'' comparison
%of the performances of the SMC, RCV and RRCV approaches.
%\end{remark}

\section{Proofs}
\label{sec:proofs}
\subsection{Proof of Theorem~\ref{thm:ChaosDecompNum}}
The proof uses the well-known fact that the system
$$
\left\{
\prod_{j=1}^J
\prod_{r=1}^m
H_{k_{j,r}}
\left(\frac{\Delta_j W^r}{\sqrt{\Delta}}\right):
k=(k_{j,r})\in\bbN_0^{J\times m}
\right\}
$$
is an orthonormal basis in $L^2(\cG_J)$,
where the $\sigma$-field $\cG_J$
is generated by the Brownian increments,
$\cG_J=\sigma(\Delta_j W:j=1,\ldots,J)$,
and goes along the lines of the proof
of Theorem~\ref{th:weak_md03}.

\subsection{Proof of Theorem~\ref{th:weak_md01}}
The proof is similar to the one of Theorem~\ref{th:weak_md03}.

\subsection{Proof of Proposition~\ref{prop:2202a1}}
Let $\cG_0$ be the trivial $\sigma$-field
and $\cG_j=\sigma(\xi_1,\ldots,\xi_j)$, $j=1,\ldots,J$.
It follows from~\eqref{eq:scheme_structure_md}
that the process $(X_{\Delta,j\Delta})_{j=0}^J$
is Markov with respect to $(\cG_j)_{j=0}^J$.
By the Markov property, we have
\begin{align*}
q_j(X_{\Delta,j\Delta})\equiv
\EE[f(X_{\Delta,T})|X_{\Delta,j\Delta}]=
\EE[f(X_{\Delta,T})|\cG_j],
\end{align*}
hence, by the tower property of conditional expectation,
\begin{align*}
q_{j-1}(x)
=\EE[q_j(\Phi_\Delta(X_{\Delta,(j-1)\Delta},\xi_j))|X_{\Delta,(j-1)\Delta}=x]
=\frac{1}{2^{m}}\sum_{y=(y^{1},\ldots,y^{m})\in\left\{ -1,1\right\} ^{m}}q_{j}(\Phi_{\Delta}(x,y)),
\end{align*}
where in the last equality we use
independence between $X_{\Delta,(j-1)\Delta}$
and~$\xi_j$.
This proves~\eqref{eq:2408a1}.
We now apply intermediate conditioning
with respect to $\cG_j$ in~\eqref{eq:coef05}
and arrive at
\begin{align*}
a_{j,r,s}(x)
=\EE\left[\left.
q_j(\Phi_\Delta(X_{\Delta,(j-1)\Delta},\xi_j)) \prod_{i=1}^r \xi_j^{s_i} 
\,\right|\,
X_{\Delta,(j-1)\Delta}=x
\right],
\end{align*}
which implies~\eqref{eq:coef05a}
due to the independence
between $X_{\Delta,(j-1)\Delta}$ and~$\xi_j$.
%Formula~\eqref{eq:coef06a} follows by a similar argument.

\subsection{Proof of Theorem~\ref{th:weak_md03}}
Let $\cG_0$ denote trivial $\sigma$-field,
and, for $j=1,\ldots,J$,
define the $\sigma$-field
$\cG_j=\sigma(\xi_1,V_1,\ldots,\xi_j,V_j)$.
Since each of the random variables
$\xi_j^r$, $j=1,\ldots,J$, $r\in\cI_1$
can take $3$ different values, each of the random variables $V_j^{kl}$, $(k,l)\in\cI_2$, can take $2$ different values and $\left|\cI_1\right|=m$, $\left|\cI_2\right|=\frac{m(m-1)}{2}$, where $\left|\cdot\right|$ means the cardinality of a set,
$L^2(\cG_J)$
is a $(3^m2^{\frac{m(m-1)}{2}})^J$-dimensional vector space.
A simple calculation reveals that,
for any fixed $j=1,\ldots,J$,
the system $\{\prod_{r\in \cI_1} H_{o_j^r}(\xi_j^r)
\prod_{(k,l)\in \cI_2} (V_j^{kl})^{s_j^{kl}}
:o_j^r\in\left\{0,1,2\right\},s_j^{kl}\in\left\{0,1\right\}\}$
is orthonormal in~$L^2(\cG_J)$.
Due to independence of
$\xi_1,V_1\ldots,\xi_J,V_J$,
the system
\begin{align}
\label{eq:ons}
\bigg\{\prod_{j=1}^J\prod_{r\in \cI_1} H_{o_j^r}(\xi_j^r)
\prod_{(k,l)\in \cI_2} (V_j^{kl})^{s_j^{kl}}
:o_j^r\in\left\{0,1,2\right\},s_j^{kl}\in\left\{0,1\right\}\bigg\}
\end{align}
is orthonormal in $L^2(\cG_J)$,
and therefore, linear independent.
The cardinality of system~\eqref{eq:ons}
is $(3^m2^{\frac{m(m-1)}{2}})^J$, i.e.\ equals the dimension
of $L^2(\cG_J)$.
Hence, linear independent system~\eqref{eq:ons}
is an orthonormal basis in $L^2(\cG_J)$.
We have $E|f(X_{\Delta,T})|^2<\infty$
because $X_{\Delta,T}$ takes finitely many values.
Therefore, $f(X_{\Delta,T})$ belongs to $L^2(\cG_J)$
and can be written
\begin{align*}
f(X_{\Delta,T})=
\sum_{\bar o\in\{0,1,2\}^{mJ}}\sum_{\bar s\in\{0,1\}^{\frac{m(m-1)}{2}J}} c_{\bar o\bar s}
\prod_{j=1}^J \prod_{r\in \cI_1} H_{o_j^r}(\xi_j^r)
\prod_{(k,l)\in \cI_2} (V_j^{kl})^{s_j^{kl}},
\end{align*}
where $\bar o=(o_1^1,\ldots,o_{J}^1,\ldots,o_1^m,\ldots,o_J^m),\bar s=(s_1^{12},\ldots,s_J^{12},s_1^{13},\ldots,s_J^{13},\ldots,s_1^{(m-1)m},\ldots,s_J^{(m-1)m})$.
Note that
$c_{\bar o\bar s}=\EE[f(X_{\Delta,T})\prod_{j=1}^J \prod_{r\in \cI_1} H_{o_j^r}(\xi_j^r)
\prod_{(k,l)\in \cI_2} (V_j^{kl})^{s_j^{kl}}]$,
in particular,
$c_{\bar0\bar 0}=\EE f(X_{\Delta,T})$.
Rearranging the terms
in the expression for $f(X_{\Delta,T})$
we rewrite it as
\begin{align}
\label{eq:repr01}
f(X_{\Delta,T})=\EE f(X_{\Delta,T})
+\sum_{j=1}^J
\sum_{(U_1,U_2)\in\cA}
\sum_{p\in\{1,2\}^{U_1}}
A_{j,p,U_1,U_2}
\prod_{r\in U_1} H_{p_r}(\xi_j^r)
\prod_{(k,l)\in U_2} V_j^{kl}
\end{align}
with $\cG_{j-1}$-measurable
random variables $A_{j,p,U_1,U_2}$. Let us now multiply both sides of the last equality
by $\prod_{r\in U_1^0} H_{p^0_r}(\xi_{j^0}^{r})
\prod_{(k,l)\in U^0_2} V_{j^0}^{kl}$,
with some $j^0\in\{1,\ldots,J\}$, $(U_1^0,U_2^0)\in\cA$, $p^0\in\left\{1,2\right\}^{U_1^0}$
and calculate conditional expectations
of the resulting expressions
given $\cG_{j_0-1}$.
Notice that,
with $j^h<j^0$ and $j^g>j^0$,
we have
\begin{align*}
&\EE[\prod_{r\in U_1^0} H_{p^0_r}(\xi_{j^0}^{r})
\prod_{(k,l)\in U^0_2} V_{j^0}^{kl}|\cG_{j^0-1}]
=\EE[\prod_{r\in U_1^0} H_{p^0_r}(\xi_{j^0}^{r})
\prod_{(k,l)\in U^0_2} V_{j^0}^{kl}]=0,\\
&\EE[A_{j^h,p,U_1,U_2}\prod_{r\in U_1} H_{p_r}(\xi_{j^h}^r)
\prod_{(k,l)\in U_2} V_{j^h}^{kl}\cdot\prod_{r\in U_1^0} H_{p^0_r}(\xi_{j^0}^{r})
\prod_{(k,l)\in U^0_2} V_{j^0}^{kl}|\cG_{j^0-1}]\\
\notag
&=A_{j^h,p,U_1,U_2}\prod_{r\in U_1} H_{p_r}(\xi_{j^h}^r)
\prod_{(k,l)\in U_2} V_{j^h}^{kl}\cdot\EE[\prod_{r\in U_1^0} H_{p^0_r}(\xi_{j^0}^{r})
\prod_{(k,l)\in U^0_2} V_{j^0}^{kl}|\cG_{j^0-1}]
=0,\\
%\end{align*}
%\begin{align*}
&\EE[A_{j^g,p,U_1,U_2}\prod_{r\in U_1} H_{p_r}(\xi_{j^g}^r)
\prod_{(k,l)\in U_2} V_{j^g}^{kl}\cdot\prod_{r\in U_1^0} H_{p^0_r}(\xi_{j^0}^{r})
\prod_{(k,l)\in U^0_2} V_{j^0}^{kl}|\cG_{j^0-1}]\\
\notag
&=\EE[\EE[A_{j^g,p,U_1,U_2}\prod_{r\in U_1} H_{p_r}(\xi_{j^g}^r)
\prod_{(k,l)\in U_2} V_{j^g}^{kl}\cdot\prod_{r\in U_1^0} H_{p^0_r}(\xi_{j^0}^{r})
\prod_{(k,l)\in U^0_2} V_{j^0}^{kl}|\cG_{j^g-1}]|\cG_{j^0-1}]=0,\\
&\EE[A_{j^0,p,U_1,U_2}\prod_{r\in U_1} H_{p_r}(\xi_{j^0}^r)
\prod_{(k,l)\in U_2} V_{j^0}^{kl}\cdot\prod_{r\in U_1^0} H_{p^0_r}(\xi_{j^0}^{r})
\prod_{(k,l)\in U^0_2} V_{j^0}^{kl}|\cG_{j^0-1}]\\
\notag
&=
A_{j^0,p,U_1,U_2}\EE[\prod_{r\in U_1} H_{p_r}(\xi_{j^0}^r)
\prod_{(k,l)\in U_2} V_{j^0}^{kl}\cdot\prod_{r\in U_1^0} H_{p^0_r}(\xi_{j^0}^{r})
\prod_{(k,l)\in U^0_2} V_{j^0}^{kl}]\\
\notag
&=A_{j^0,p,U_1,U_2}\delta_{p,p^0}\delta_{U_1,U_1^0}\delta_{U_2,U_2^0},
\end{align*}
where $\delta_{\cdot,\cdot}$ is the Kronecker delta.
Thus, the coefficients
$A_{j,p,U_1,U_2}$ in~\eqref{eq:repr01} are given by
\begin{align}
\label{eq:1902a1}
A_{j,p,U_1,U_2}=\EE[f(X_{\Delta,T})
\prod_{r\in U_1} H_{p_r}(\xi_j^r)
\prod_{(k,l)\in U_2} V_j^{kl} | \cG_{j-1}].
\end{align}
Let us now prove that
\begin{align}
\label{eq:1902a2}
\EE[f(X_{\Delta,T})
\prod_{r\in U_1} H_{p_r}(\xi_j^r)
\prod_{(k,l)\in U_2} V_j^{kl} | \cG_{j-1}]
=\EE[f(X_{\Delta,T})
\prod_{r\in U_1} H_{p_r}(\xi_j^r)
\prod_{(k,l)\in U_2} V_j^{kl} | X_{\Delta,(j-1)\Delta}].
\end{align}
In what follows we use the functions $q_j$ from \eqref{tower:q} and notice that,
by the Markov property of
$(X_{\Delta,j\Delta})_{j=0,\ldots,J}$
with respect to $(\cG_j)$,
which is due to~\eqref{eq:2002a5},
we also have
\begin{align}
\label{eq:1902a3}
q_j(X_{\Delta,j\Delta})
=\EE[f(X_{\Delta,T})|\cG_j].
\end{align}
Let us set
\begin{align}
\label{eq:1902a5}
h(X_{\Delta,(j-1)\Delta},\xi_j,V_j)
=\prod_{r\in U_1} H_{p_r}(\xi_j^r)
\prod_{(k,l)\in U_2} V_j^{kl}q_j(X_{\Delta,j\Delta})
\end{align}
and notice that,
due to~\eqref{eq:2002a5},
this is indeed a function of
$X_{\Delta,(j-1)\Delta}$, $\xi_j$ and $V_j$ only. Further, let us set
\begin{align}
\label{eq:1902a8}
g\left(x\right)=\mathbb{E}\left[h\left(x,\xi_j,V_j\right)\right].
\end{align}
Using the tower property of conditional expectations
together with~\eqref{eq:1902a3},~\eqref{eq:1902a5} and ~\eqref{eq:1902a8}, we get
\begin{align}
\EE[f(X_{\Delta,T})
\prod_{r\in U_1} H_{p_r}(\xi_j^r)
\prod_{(k,l)\in U_2} V_j^{kl} |\cG_{j-1}]
&=\EE[\prod_{r\in U_1} H_{p_r}(\xi_j^r)
\prod_{(k,l)\in U_2} V_j^{kl}\EE[f(X_{\Delta,T})|\cG_j]\,|\cG_{j-1}]
\notag
\\
&=\EE[h(X_{\Delta,(j-1)\Delta},\xi_j,V_j)|\cG_{j-1}]
=g(X_{\Delta,(j-1)\Delta}),
\label{eq:1902a6}
\end{align}
where the last equality is due to the facts
that $X_{\Delta,(j-1)\Delta}$
is $\cG_{j-1}$-measurable
and the pair $(\xi_j,V_j)$ is independent of $\cG_{j-1}$. Moreover, applying \eqref{eq:1902a6}, we also obtain
\begin{align}
\label{eq:1902a7}
\EE[f(X_{\Delta,T})
\prod_{r\in U_1} H_{p_r}(\xi_j^r)
\prod_{(k,l)\in U_2} V_j^{kl}|X_{\Delta,(j-1)\Delta}]
=\EE[g(X_{\Delta,(j-1)\Delta})|X_{\Delta,(j-1)\Delta}]=g(X_{\Delta,(j-1)\Delta}).
\end{align}
Comparing~\eqref{eq:1902a6} and~\eqref{eq:1902a7},
we arrive at~\eqref{eq:1902a2}.
Together with~\eqref{eq:1902a1} and~\eqref{eq:repr01},
this proves~\eqref{eq:2002a1} and~\eqref{eq:2002a2}.

\subsection{Proof of Proposition~\ref{prop:0403a1}}
The proof is similar to the one of Proposition~\ref{prop:2202a1}.

\subsection{Proof of Theorem~\ref{th:2104a1}}
For the proof, we need the following
multivariate gene\-ra\-li\-sa\-tion of
Lemma~11.1 in~\cite{gyorfi2002distribution}.

\begin{lemma}
\label{gjlemma}
Let $a\colon\left[0,1\right]^d\to\mathbb{R}$ be a
$\left(p+1,C\right)$-smooth function w.r.t.\ the norm
$\left|\cdot\right|_h$, where $d\in\mathbb{N}$, $h\in[1,\infty]$ and $p\in\mathbb{N}_0$. Further, let $g$ be a piecewise polynomial of degree less than or equal to $p$ w.r.t.\ an equidistant partition of $\left[0,1\right]^d$ in $Q^d$ cubes. Then it holds
\begin{align}
\label{gjmuldim}
\sup_{x\in\left[0,1\right]^d}\left| a\left(x\right)-g\left(x\right)\right|&\leq\frac{C}{d^{1-1/h}\left(p+1\right)!}\left(\frac{d}{2Q}\right)^{p+1}.
\end{align}
\end{lemma}

\begin{proof}
Consider the Taylor expansion of the function $a$ up to the degree $p$ around $z\in\left(0,1\right)^d$:
\begin{align*}
a_p\left(x\right)&=\sum_{n=0}^p\frac{1}{n!}\sum_{l_1+\ldots+l_d=n}\binom{n}{l_1,\ldots,l_d}\frac{\partial^n m\left(z\right)}{\partial x_1^{l_1}\cdots\partial x_d^{l_d}}\prod_{i=1}^d\left(x_i-z_i\right)^{l_i}.
\end{align*}
The remainder term has the form
$$
a\left(x\right)-a_p\left(x\right)=\frac{1}{p!}\int\limits_0^1\left(1-t\right)^p\sum_{l_1+\ldots+l_d=p+1}\binom{p+1}{l_1,\ldots,l_d}\frac{\partial^{p+1} a\left(z+t\left(x-z\right)\right)}{\partial x_1^{l_1}\cdots\partial x_d^{l_d}}\prod_{i=1}^d\left(x_i-z_i\right)^{l_i}dt.
$$
At first, we will focus on the case $p>0$. For $g=a_p$ we have
\begin{align*}
&a\left(x\right)-g\left(x\right)=a\left(x\right)-a_{p-1}\left(x\right)-\frac{1}{p!}\sum_{l_1+\ldots+l_d=p}\binom{p}{l_1,\ldots,l_d}\frac{\partial^p a\left(z\right)}{\partial x_1^{l_1}\cdots\partial x_d^{l_d}}\prod_{i=1}^d\left(x_i-z_i\right)^{l_i}\\
\notag
=&\frac{1}{\left(p-1\right)!}\int\limits_0^1\left(1-t\right)^{p-1}\sum_{l_1+\ldots+l_d=p}\binom{p}{l_1,\ldots,l_d}\left(\frac{\partial^{p} a\left(z+t\left(x-z\right)\right)}{\partial x_1^{l_1}\cdots\partial x_d^{l_d}}-\frac{\partial^p a\left(z\right)}{\partial x_1^{l_1}\cdots\partial x_d^{l_d}}\right)\\
\notag
\phantom{=}&\times\prod_{i=1}^d\left(x_i-z_i\right)^{l_i}dt.
\end{align*}
Since $a$ is $\left(p+1,C\right)$-smooth, we obtain
\begin{align*}
\notag
\left| a\left(x\right)-g\left(x\right)\right|&\leq \frac{C}{\left(p-1\right)!}\left| x-z\right|_h\int\limits_0^1t\left(1-t\right)^{p-1}dt
\sum_{l_1+\ldots+l_d=p}\binom{p}{l_1,\ldots,l_d}\prod_{i=1}^d\left|x_i-z_i\right|^{l_i}\\
\notag
&=\frac{C}{\left(p+1\right)!}\left|x-z\right|_h
\,\left(\sum_{i=1}^d\left|x_i-z_i\right|\right)^p
\leq \frac{C}{\left(p+1\right)!}\left|x-z\right|_h^{p+1}
d^{p\left(1-1/h\right)}.
%\label{intcomp}
\end{align*}
In the case $p=0$ this inequality holds, too.
This follows directly from the
$\left(p+1,C\right)$-smoothness assumption.

Next, we consider the equidistant partitioning
of $[0,1]^d$ into $Q^d$ cubes $K^1,\ldots,K^{Q^d}$ with $\bigcup_{k=1}^{Q^d}K^k=\left[0,1\right]^d$.
Let $z^k$ be the midpoint of $K^k$.
We then have
$\sup_{x\in K^k}\left|x-z^k\right|_h=\frac{d^{1/h}}{2Q}$
for all $k\in\left\{1,\ldots,Q^d\right\}$.
This finally yields~\eqref{gjmuldim}.
\end{proof}

We now proceed with the proof of Theorem~\ref{th:2104a1}.
Define the set 
$$
\Psi^{Q,p}\doteq\text{span}\left(\left\{\psi^{k,1},\ldots,\psi^{k,n}\colon k\in\left\{1,\ldots,Q^d\right\},n=\binom{p+d}{d}\right\}\right).
$$
We would like to apply
Theorem~11.3 in~\cite{gyorfi2002distribution},
which gives us 
\begin{align}
\notag
\EE\|\tilde a_{j,o,U_1,U_2}-a_{j,o,U_1,U_2}\|^2_{L^2(\PP_{\Delta,j-1})}
&\le\tilde c\max\left\{\Sigma^2,A^2\,\Delta\right\}(\log N+1)\frac{\binom{p+d}d Q^d}{N}\\
\label{gjorfi_direct}
&\hspace{1em}+8\inf_{g\in\Psi^{Q,p}}\int_{\mathbb{R}^d}\left(a_{j,o,U_1,U_2}\left(x
\right)-g\left(x\right)\right)^2\,\P_{\Delta,j-1}(dx).
\end{align}
However, the maximum in~\eqref{gjorfi_direct} is in fact a sum of two terms $A^2\,\Delta(\log N+1)$ and $\Sigma^2$ so that the logarithm is only included in one term
(see proof of Theorem~11.3 in~\cite{gyorfi2002distribution}).
Next, we split the integral in~\eqref{gjorfi_direct} into two parts:
\begin{align}
\notag
\int_{\mathbb{R}^d}\left(a_{j,o,U_1,U_2}\left(x
\right)-g\left(x\right)\right)^2\,\P_{\Delta,j-1}(dx)&=\int_{\left[-R,R\right]^d}\left(a_{j,o,U_1,U_2}\left(x
\right)-g\left(x\right)\right)^2\,\P_{\Delta,j-1}(dx)\\
\label{int_split}
&\phantom{=}+\int_{\mathbb{R}^d\setminus\left[-R,R\right]^d}a_{j,o,U_1,U_2}^2\left(x
\right)\,\P_{\Delta,j-1}(dx),
\end{align}
since $g\left(x\right)=0$ for $x\notin\left[-R,R\right]^d$ for $g\in\Psi^{Q,p}$. The second integral in~\eqref{int_split} refers to the case $|X_{\Delta,(j-1)\Delta}|_\infty> R$, where we simply use Assumptions~(A2) and~(A4) to get
\begin{align*}
\notag
\int_{\mathbb{R}^d\setminus\left[-R,R\right]^d}a_{j,o,U_1,U_2}^2\left(x
\right)\,\P_{\Delta,j-1}(dx)&\leq \sup\limits_{x\in\R^d} |a_{j,o,U_1,U_2}(x)|^2\:\PP(|X_{\Delta,(j-1)\Delta}|_\infty>R)\\
&\leq A^2\,\Delta B_\nu R^{-\nu}.
\end{align*}
Regarding the first integral in~\eqref{int_split},
we obtain by Lemma~\ref{gjlemma}
\begin{align*}
\notag
\inf_{g\in\Psi^{Q,p}}\int_{\left[-R,R\right]^d}\left(a_{j,o,U_1,U_2}\left(x
\right)-g\left(x\right)\right)^2\,\P_{\Delta,j-1}(dx)&\leq \inf_{g\in\Psi^{Q,p}}\sup_{x\in\left[-R,R\right]^d}\left| a_{j,o,U_1,U_2}\left(x\right)-g\left(x\right)\right|^2\\
&\leq\frac{C_h^2}{d^{2-2/h}\left(p+1\right)!^2}\left(\frac{Rd}{Q}\right)^{2p+2}
\end{align*}
(notice that, since we consider $\left[-R,R\right]^d$
instead of $[0,1]^d$,
the expression $\frac{d}{2Q}$ in~\eqref{gjmuldim}
is replaced by $\frac{Rd}{Q}$ because
$\sup_{x\in K^k}\left|x-z^k\right|_h=\frac{Rd^{1/h}}{Q}$
with $z^k$ being the midpoint of~$K^k$).

\subsection{Proof of Theorem~\ref{compl_rcv}}
Let us
%for simplicity, first only consider the terms w.r.t. the variables $J,N,N_0,Q,R$ which shall be optimised, since the constant $B_\nu$ does not affect the terms on $\varepsilon$. Further, we 
consider the log-cost and log-constraints rather
than~\eqref{cost_rcv} and~\eqref{constr_rcv}.
Further, let us subdivide the optimisation problem into two cases:

\noindent
\begin{enumerate}
\item $N\lesssim N_0$.
This gives us the Lagrange function
\begin{align}
\label{lagrange_strong}
L_{\lambda_1,\ldots,\lambda_5}(J,N,N_0,Q,R)\doteq & \log(J)+\log(N_0)+d\log(Q)+\lambda_1(-4\log(J)-2\log(\varepsilon))\\
\notag
&+\lambda_2(\log(J)+d\log(Q)-\log(N)
-\log(N_0)-2\log(\varepsilon))\\
\notag
&+\lambda_3(\log(J)+2(p+1)(\log(R)-\log(Q))-\log(N_0)-2\log(\varepsilon))\\
\notag
&+\lambda_4(-\nu\log(R)-\log(N_0)
-2\log(\varepsilon))+\lambda_5(\log(N)-\log(N_0)),
\end{align}
where $\lambda_1,\ldots,\lambda_5\ge 0$. Thus, considering of the conditions $\frac{\partial L}{\partial J}=\frac{\partial L}{\partial N}=\frac{\partial L}{\partial N_0}=\frac{\partial L}{\partial Q}=\frac{\partial L}{\partial R}\stackrel{!}{=}0$ gives us the following Lagrange parameters
\begin{align*}
\lambda_1&=\frac{3d\nu+6\nu(p+1)}{4(d\nu+2(p+1)(2\nu+d))},\\
\lambda_2&=\frac{2(p+1)(\nu-d)-d\nu}{d\nu+2(p+1)(2\nu+d)}=\lambda_5,\\
\lambda_3&=\frac{3d\nu}{d\nu+2(p+1)(2\nu+d)},\\
\lambda_4&=\frac{6d(p+1)}{d\nu+2(p+1)(2\nu+d)}.
\end{align*}
Obviously it holds $\lambda_1,\lambda_3,\lambda_4>0$, so that we can deduce
$$
J\asymp\varepsilon^{-\frac{1}{2}},\quad
R\asymp\left(Q^{4(p+1)}\varepsilon\right)^{\frac{1}{2\nu+4(p+1)}},\quad
N_0\asymp\left(Q^{4\nu(p+1)}\varepsilon^{5\nu+8(p+1)}\right)^{-\frac{1}{2\nu+4(p+1)}}.
$$
Regarding $\lambda_2$ ($\equiv\lambda_5$), we have to consider two cases again.

\noindent
\begin{itemize}
\item Case~(1a): $\lambda_2=\lambda_5=0$. From this condition, we get
$\nu=\frac{2 d\left(p+1\right)}{2\left(p+1\right)-d}$
and
$p>\frac{d-2}{2}$.
(The latter guarantees that $\nu$ is positive.)
Thus,
\begin{align*}
R\asymp\left(Q^{4(p+1)}\varepsilon\right)^{\frac{2(p+1)-d}{8(p+1)^2}},\quad 
N_0\asymp\varepsilon^{-\frac{8(p+1)+d}{4(p+1)}}Q^{-d}.
\end{align*}
Hence, the complexity $JQ^dN_0\asymp\varepsilon^{-\frac{10(p+1)+d}{4(p+1)}}\gtrsim\varepsilon^{-2.5}$ is worse than
that of the SMC in this case.

\item Case~(1b): $\lambda_2=\lambda_5>0$.
This implies
$\nu>\frac{2 d\left(p+1\right)}{2\left(p+1\right)-d}$
and
$p>\frac{d-2}{2}$.
(Again, the second condition guarantees that $\nu$ is positive.)
We can deduce
$$
Q\asymp \varepsilon^{-\frac{5\nu+6(p+1)}{2d\nu+4(p+1)(2\nu+d)}}, \quad R\asymp \varepsilon^{-\frac{6(p+1)-d}{2d\nu+4(p+1)(2\nu+d)}},\quad
N_0\asymp \varepsilon^{-\frac{5d\nu+2(p+1)(5\nu+4d)}{2d\nu+4(p+1)(2\nu+d)}}\asymp N, 
$$
so that the complexity 
\begin{align}
\label{compl_rcv_eps}
JQ^dN_0\asymp \varepsilon^{-\frac{11d\nu+2(p+1)(7\nu+8d)}{2d\nu+4(p+1)(2\nu+d)}}
\end{align}
is a better solution than that in case~(1a).
\end{itemize}

\medskip
\item $N\gtrsim N_0$.
This gives us the Lagrange function
\begin{align*}
\tilde L_{\lambda_1,\ldots,\lambda_5}(J,N,N_0,Q,R)\doteq & \log(J)+\log(N)+d\log(Q)+\lambda_1(-4\log(J)-2\log(\varepsilon))\\
\notag
&+\lambda_2(\log(J)+d\log(Q)-\log(N)
-\log(N_0)-2\log(\varepsilon))\\
\notag
&+\lambda_3(\log(J)+2(p+1)(\log(R)-\log(Q))-\log(N_0)-2\log(\varepsilon))\\
\notag
&+\lambda_4(-\nu\log(R)-\log(N_0)
-2\log(\varepsilon))+\lambda_5(\log(N_0)-\log(N)).
\end{align*}
Analogously to the procedure above we get the same optimal solution~\eqref{compl_rcv_eps}.
\end{enumerate}
%Now we consider also the constant $B_\nu$ and obtain~\eqref{sol_rcv}--\eqref{eq:compld} via equalising all constraints in~\eqref{constr_rcv} as well as considering $N\asymp N_0$.

\medskip
Thus, we arrive at~\eqref{sol_rcv} and~\eqref{eq:compld},
provided that $p>\frac{d-2}2$ and $\nu>\frac{2d(p+1)}{2(p+1)-d}$.
Let us finally prove the statement in
footnote~\ref{fn_rcv} on page~\pageref{fn_rcv},
i.e.\ that the complexity of the RCV approach
would be worse than that of the SMC
whenever at least one of the above inequalities is violated.
More precisely, the statement we are going to prove
sounds as follows.
If either $p\le\frac{d-2}2$ (recall that $p\in\bbN_0$)
or $\nu\le\frac{2(p+1)d}{2(p+1)-d}$ (recall that $\nu>0$),
then the cost $\cC$ of the RCV algorithm
given in~\eqref{cost_rcv}
is worse than $\eps^{-2.5}$
regardless of the choice of
$J$, $Q$, $R$, $N$ and $N_0$
such that~\eqref{constr_rcv} holds true.

We first remark that any choice of
$J$, $Q$, $R$, $N$, $N_0$
such that $R$ does not tend to infinity as $\eps\searrow0$
results in $\cC\gtrsim\eps^{-2.5}$.
Indeed, in this case we see from the first and the fourth terms
in~\eqref{constr_rcv}
that $J\gtrsim\eps^{-0.5}$ and $N_0\gtrsim\eps^{-2}$,
hence $\cC\gtrsim JN_0\gtrsim\eps^{-2.5}$.
Therefore, below we consider without loss of generality
only such choices of $J$, $Q$, $R$, $N$, $N_0$,
where $R$ tends to infinity as $\eps\searrow0$,
and discuss the following two cases.

Let $p\le \frac{d-2}{2}$, that is, $2(p+1)\le d$.
Then we obtain from the third term in~\eqref{constr_rcv} 
$$
Q^dN_0\gtrsim Q^{2(p+1)}N_0\gtrsim\varepsilon^{-2}JR^{2(p+1)}\gtrsim \varepsilon^{-2}J
$$
and hence, together with $J\gtrsim\varepsilon^{-0.5}$ (see the first term in~\eqref{constr_rcv}), we have for the cost
$$
\cC\gtrsim JQ^dN_0\gtrsim\eps^{-2} J^2\gtrsim\eps^{-3},
$$
which is even worse than $\eps^{-2.5}$.

Finally, let $p>\frac{d-2}{2}$, that is, $2(p+1)>d$,
and $0<\nu\le\frac{2(p+1)d}{2(p+1)-d}$.
Then we get from the third and the fourth terms
in~\eqref{constr_rcv}
\begin{gather*}
R^{2(p+1)}\lesssim J^{-1}Q^{2(p+1)}N_0\varepsilon^2,\\
R^\frac{2(p+1)d}{2(p+1)-d}\gtrsim R^\nu\gtrsim N_0^{-1}\varepsilon^{-2}.
\end{gather*}
Therefore,
$$
J^{-\frac d{2(p+1)-d}}
Q^{\frac{2(p+1)d}{2(p+1)-d}}
N_0^{\frac d{2(p+1)-d}}
\eps^{\frac{2d}{2(p+1)-d}}
\gtrsim
N_0^{-1}\varepsilon^{-2}.
$$
This yields
$$
J^{-\frac d{2(p+1)-d}}
Q^{\frac{2(p+1)d}{2(p+1)-d}}
N_0^{\frac{2(p+1)}{2(p+1)-d}}
\gtrsim
\varepsilon^{-\frac{4(p+1)}{2(p+1)-d}},
$$
and we deduce
$$
J^{-\frac{d}{2(p+1)}}Q^dN_0\gtrsim \varepsilon^{-2}.
$$
Together with $J\gtrsim\varepsilon^{-0.5}$,
we obtain for the cost
$$
\cC\gtrsim JQ^dN_0\gtrsim
J^{1+\frac d{2(p+1)}}\eps^{-2}\gtrsim\eps^{-2.5},
$$
which concludes the proof.

\subsection{Proof of Theorem~\ref{compl_rrcv}}
The proof is similar to the one of Theorem~\ref{compl_rcv}.

\bibliographystyle{abbrv}
\bibliography{refs}
\end{document}